\documentclass{amsart}
\usepackage{amsmath, amssymb, amsthm, graphicx}
\usepackage{geometry}
\usepackage[all]{xy}
\usepackage{xcolor}
\usepackage{mathtools}

\theoremstyle{thmstyle}
\newtheorem{thm}{Theorem}[section]
\newtheorem{theorem}[thm]{Theorem}

\newtheorem{lemma}[thm]{Lemma}

\newtheorem{prop}[thm]{Proposition}
\newtheorem{proposition}[thm]{Proposition}

\newtheorem{corollary}[thm]{Corollary}

\newtheorem{conjecture}{Conjecture}[section]
\newtheorem{question}{Question}[section]

\theoremstyle{exstyle}
\newtheorem{problem}{Exercise}[]

\newtheorem{example}[thm]{Example}

\theoremstyle{defstyle}
\newtheorem{defn}{Definition}[section]
\newtheorem{definition}[thm]{Definition}

\newtheorem{def-prop}[thm]{Definition-Proposition}
\newtheorem{def-lem}[thm]{Definition-Lemma}
\newtheorem{rem-convention}[thm]{Remark-Convention}
\newtheorem{def-note}[thm]{Definition-Notation}

\theoremstyle{remstyle}

\newtheorem{remark}[thm]{Remark}

\theoremstyle{remstyle}

\newcommand{\Hom}{\operatorname{Hom}}
\newcommand{\Fac}{\operatorname{Fac}}
\newcommand{\Ext}{\operatorname{Ext}}

\newcommand{\taurigid}{\operatorname{\tau-rigid}}

\DeclareMathOperator*{\rad}{rad}

\DeclareMathOperator*{\modu}{mod}
\DeclareMathOperator*{\Modu}{Mod}
\DeclareMathOperator*{\proj}{proj}

\newcommand{\Z}{\mathcal{Z}}

\DeclareMathOperator*{\tors}{tors}

\DeclareMathOperator*{\wide}{wide}

\DeclareMathOperator*{\End}{End}

\DeclareMathOperator*{\Filt}{Filt}

\DeclareMathOperator*{\brick}{brick}

\DeclareMathOperator*{\pd}{pd}

\DeclareMathOperator*{\D}{D}

\DeclareMathOperator*{\GL}{GL}
\DeclareMathOperator*{\SL}{SL}
\DeclareMathOperator*{\Aut}{Aut}
\DeclareMathOperator*{\rep}{rep}

\DeclareMathOperator*{\Irr}{Irr}

\DeclareMathOperator*{\coker}{coker}

\DeclareMathOperator*{\Mat}{Mat}
\DeclareMathOperator*{\stab}{stab}
\DeclareMathOperator*{\s-brick}{s-brick}
\DeclareMathOperator*{\simp}{simp}
\DeclareMathOperator*{\SI}{SI}

\makeatletter
\newcommand*{\doublerightarrow}[2]{\mathrel{
  \settowidth{\@tempdima}{$\scriptstyle#1$}
  \settowidth{\@tempdimb}{$\scriptstyle#2$}
  \ifdim\@tempdimb>\@tempdima \@tempdima=\@tempdimb\fi
  \mathop{\vcenter{
    \offinterlineskip\ialign{\hbox to\dimexpr\@tempdima+1em{##}\cr
    \rightarrowfill\cr\noalign{\kern.5ex}
    \rightarrowfill\cr}}}\limits^{\!#1}_{\!#2}}}

\title{On the bricks (Schur representations) of finite dimensional algebras}

\author[Kaveh Mousavand, Charles Paquette]{Kaveh Mousavand, Charles Paquette} 

\address{Kaveh Mousavand: Representation Theory and Algebraic Combinatorics Unit, Okinawa Institute of Science and Technology (OIST), Japan}
\email{mousavand.kaveh@gmail.com}
\address{Charles Paquette: Department of Mathematics and Computer Science, Royal Military College of Canada, Kingston ON, Canada}
\email{charles.paquette.math@gmail.com}


\begin{document}
\maketitle

\begin{abstract}
This manuscript treats the diverse applications of bricks within modern representation theory and several related domains, and reviews the recent developments and new results on bricks (a.k.a Schur representations). 
The current survey is an extended version of a mini-course by the second-named author, delivered in the research school on ``New Developments in Representation Theory of Algebras", held in November of 2024, at Okinawa Institute of Science and Technology (OIST), Japan. The review is mainly oriented towards the direction of research developed by the authors, which has evolved around the algebraic and geometric properties of bricks. More specifically, we discuss the emergence of bricks in $\tau$-tilting theory, torsion theory, geometric representation theory and invariant theory, while providing some links between those. Although we review the applications and properties of bricks from many different areas, the article is not meant to be an exhaustive survey on bricks in representation theory. In the setting of finite dimensional algebras over an algebraically closed field, this manuscript (and many of the recent works of the authors) is strongly motivated by an open conjecture originally posed by the first-named author in 2019, the so-called \emph{second brick Brauer-Thrall conjecture}. In the later sections, where the main focus is on the tame algebras and some other new notions of tameness, we prove some new results on the aforementioned conjecture.
\end{abstract}

\tableofcontents

\section*{Notations and conventions}

Throughout the entire manuscript, we let $k$ denote an algebraically closed field and, unless specified otherwise, $A$ is always assumed to be a finite dimensional basic (associative and unital) $k$-algebra. With no loss of generality, and by Gabriel's Theorem, we may assume that $A = kQ/I$ where $Q$ is a finite quiver and $I$ is an admissible ideal of $kQ$. A convenient numbering, which we will generally assume, is $Q_0 = \{1,2,\ldots,n\}$, where $Q_0$ denotes the set of vertices of $Q$.

We denote by $\Modu A$ the category of all left (unitary) $A$-modules, and let $\modu A$ denote the full subcategory of $\Modu A$ of the finite dimensional modules. 
Thanks to the Krull-Remak-Schmidt Theorem, every $M \in \modu A$ has a unique decomposition into a finite direct sum of indecomposable representations (up to isomorphism and ordering). We let $|M|$ denote the number of non-isomorphic indecomposable direct summands of $M$.

In the following, we only recall some terminology and fix the notation that we use throughout the notes, and for all the standard materials and rudiments of representation theory of associative algebras, we refer to \cite{ASS} and references therein.

\section{Preliminaries and background}

\subsection{Projectives, injectives and simples.} For each vertex $i$ of $Q$, we let $e_i$ denote the class of the trivial path at $i$, modulo the ideal $I$. We note that $\{e_1, \ldots, e_n\}$ forms a complete set of primitive orthogonal idempotents of $A$. Up to isomorphisms, the indecomposable non-isomorphic projective modules in $\modu A$ are
$$Ae_1, Ae_2, \ldots, Ae_n$$
where we will simply denote $Ae_i$ by $P_i$. Let $D = \Hom_k(-,k)$ denote the duality between $\modu A$ and $\modu(A^{\rm op})$. The indecomposable non-isomorphic projective modules in $\modu(A^{\rm op})$ are
$e_1A, e_2A, \ldots, e_nA$ and it follows that the indecomposable non-isomorphic injective modules in $\modu A$ are
$$D(e_1A), D(e_2A), \ldots, D(e_nA)$$
and we simply denote $D(e_iA)$ by $I_i$. Finally, there are exactly $n$ non-isomorphic simple $A$-modules, which are the tops of the non-isomorphic indecomposable projective modules. We let $S_i$ denote the top of $P_i$. We have $S_i = Ae_i/{\rm rad}(A)e_i$, or equivalently, $S_i$ is the representation of the bound quiver $(Q,I)$ where the vector space at vertex $i$ is $k$ while all other vector spaces (and all maps) are zero.

\subsection{Auslander-Reiten translations.} We need some basic tools from the Auslander-Reiten theory, which we briefly recall (for details, see \cite[Chapter IV]{ASS}). We denote by $\nu$ the Nakayama functor $\nu : = D\Hom(-,A)$, which is a covariant functor from $\modu A$ to $\modu A$. Given $M \in \modu A$, one takes any projective presentation 
$$P_{-1} \stackrel{f}{\to} P_0 \to M \to 0 $$
of $M$ and considers the kernel $\tau M$ of $\nu(f)$:
$$0 \to \tau M \to \nu(P_{-1}) \stackrel{\nu(f)}{\to} \nu(P_0).$$
It turns out that $\tau M$ is well-defined up to injective direct summands. In fact, $\tau$ gives rise to an equivalence
$$\tau: \underline{\modu} A \to \overline{\modu}A,$$
from the projective-stable module category $\underline{\modu} A$ to the injective-stable module category $\overline{\modu}A$. This functor is called the \emph{Auslander-Reiten translate}. We recall that $\underline{\modu}A$ is defined to be the quotient category of $\modu A$ by the ideal of morphims that factor through a projective module. Dually, $\overline{\modu}A$ is defined to be the quotient category of $\modu A$ by the ideal of morphims that factor through an injective module. The quasi-inverse $\tau^{-1}$ of $\tau$ is the \emph{inverse Auslander-Reiten translate}. 

\begin{remark}
    We note that although $\tau$ and $\tau^{-1}$, as functors, are functors between the stable module categories, we can still apply them to indecomposable objects in $\modu A$. If $M$ is indecomposable non-projective, then $\tau M$ is indecomposable non-injective. Similarly, if $M$ is indecomposable non-injective, then $\tau^{-1} M$ is indecomposable non-projective. If $M$ is projective, then we let $\tau M=0$. Dually, if $M$ is injective, then $\tau^{-1} M=0$.
\end{remark}

Some important properties of the Auslander-Reiten translates that we freely use are recalled below.

\begin{prop}
    We have functorial isomorphisms, in $X$ and in $Y$, as follows:
    $$\Ext^1(X,Y) \cong D \overline{\Hom}(Y, \tau X)$$
    $$\Ext^1(X,Y) \cong D \underline{\Hom}(\tau^{-1}Y,X)$$
    where $\overline{\Hom}$ is the Hom bifunctor within $\overline{\modu} A$, and $\underline{\Hom}$ is the Hom bifunctor within $\underline{\modu} A$.
\end{prop}

\section{Bricks, $\tau$-tilting theory, and torsion theory}

We begin this section by recalling the definition of the main objects of interest in this manuscript (i.e., bricks) and some remarks on these important modules, which are also known as Schur-representations.

\begin{defn}
    Let $M \in \modu A$. Then $M$ is a \emph{brick} provided $\End_A(M)$ is a division $k$-algebra.
\end{defn}

\begin{problem}
    If $M \in \modu A$, prove that $M$ is a brick if and only if $\End_A(M) = \{\lambda 1_M \mid \lambda \in k\}$, if and only if $\End_A(M)$ is one dimensional, as a $k$-vector space. Show that if the assumption that $k$ is algebraically closed is dropped, then for a brick $M$ in $\modu A$, the algebra $\End_A(M)$ need not be one dimensional.
\end{problem}

\begin{remark} To make some comparison with arbitrary indecomposables, we make a few remarks on some properties of bricks.
\begin{enumerate}
    \item[$(a)$] Any brick $M$ is indecomposable. Indeed, $\End_A(M)$ being a division algebra only has $0$ and $1_M$ as indempotents. This forces $M$ to be indecomposable.
    \item[$(b)$] Each simple module is a brick, but any non-local (connected) algebra has some non-simple bricks. Treated through \emph{wide} subcategories, each brick in $\modu A$ can be seen as a simple object in a wide subcategory of $\modu A$ (see Proposition \ref{Prop: Ringle-wide-semibrick}).
    \item[$(c)$] Every thin indecomposable module is a brick. Recall that $M$ is said to be \emph{thin} if $e_iM$ is at most one dimensional, for every $i \in Q_0$.
    \item[$(d)$] In general, over an arbitrary algebra, not every indecomposable is a brick. It follows from \cite{Dr} that all indecomposable $A$-modules are bricks if and only if the algebra $A$ is locally representation directed. Every such algebra is necessarily representation-finite.
    \item[$(e)$] The projective module $P_i$ is a brick if and only if $e_iAe_i$ is one dimensional, which is the case if and only if every path of positive length from $i$ to $i$ is in $I$.
\end{enumerate}
\end{remark}

Bricks are important objects that govern many aspect of torsion theory, $\tau$-tilting theory and silting theory. In this section, we briefly recall some key notions from $\tau$-tilting theory and torsion theory. We  then explain how bricks naturally arise as part of these theories. We start with the definition of $\tau$-rigidity and of (support) $\tau$-tilting modules, as formally introduced by Adachi, Iyama and Reiten \cite{AIR}.

\begin{defn}
    Let $M \in \modu A$. 
    \begin{enumerate}
        \item[$(1)$] The module $M$ is \emph{$\tau$-rigid} provided $\Hom_A(M, \tau M)=0$.
        \item[$(2)$] The module $M$ is \emph{$\tau$-tilting} provided it is $\tau$-rigid and $|M|=n$.
        \item[$(3)$] The module $M$ is \emph{support $\tau$-tilting} provided there is an idempotent $e \in A$ such that $M$ is $\tau$-tilting over $A/AeA$.
    \end{enumerate}
\end{defn}

\begin{remark}
    \begin{enumerate}
        \item[$(a)$] It follows from a result of Skowronski (see \cite[VIII, Lemma 5.3]{ASS}) that if $M$ is $\tau$-rigid, then $|M| \le n$. Hence, $\tau$-tilting modules are the $\tau$-rigid modules with the maximal possible number of non-isomorphic direct summands.
        \item[$(b)$] If $eM=0$, then $M$ is $\tau$-rigid over $A/AeA$ if and only if it is $\tau$-rigid over $A$.
        \item[$(c)$] Recall that $M$ is \emph{partial tilting} provided that $\Ext^1_A(M,M)=0$ and the projective dimension of $M$ is at most one. It is \emph{tilting} if, in addition, we have $|M|=n$. Since the projective dimension of $M$ is at most one, it is not hard to check that $\overline{\Hom}_A(M, \tau M) = \Hom_A(M, \tau M)$, as vector spaces. Hence, rigidity and $\tau$-rigidity are equivalent notions in this case. Thus, every tilting module is $\tau$-tilting. Conversely, every faithful $\tau$-rigid module is tilting. This follows from the following two facts. First,  $M$ is faithful if and only if there is an epimorphism $M^r \to D(A)$ for some $r$. Second, a module $M$ has projective dimension at most one if and only if $\Hom_A(D(A), \tau M)=0$ (see \cite[VIII, Lemma 5.1]{ASS}).  
    \end{enumerate}
\end{remark}

Now, we recall the notion of torsion pairs. In these notes, all subcategories of $\modu A$ are full, closed under finite direct sums and under direct summands and isomorphisms.

\begin{defn}
    Let $\mathcal{T}, \mathcal{F}$ be two subcategories of $\modu A$. Then 
    \begin{enumerate}
        \item[$(1)$] $\mathcal{T}$ is a torsion class if it is closed under quotients and under extensions. Equivalently, for any short exact sequence $0 \to L \to M \to N \to 0$
        in $\modu A$, we have that
        \begin{itemize}
            \item If $M \in \mathcal{T}$, then $N \in \mathcal{T}$.
            \item If $L, N \in \mathcal{T}$, then $M \in \mathcal{T}$.
        \end{itemize}
        \item[$(2)$] $\mathcal{F}$ is a torsion-free class if it is closed under submodules and under extensions. Equivalently, for any short exact sequence $0 \to L \to M \to N \to 0$ in $\modu A$, we have that
        \begin{itemize}
            \item If $M \in \mathcal{F}$, then $L \in \mathcal{F}$.
            \item If $L, N \in \mathcal{F}$, then $M \in \mathcal{F}$.
        \end{itemize}
        \item[$(3)$] The pair $(\mathcal{T}, \mathcal{F})$ is a torsion pair provided
        \begin{itemize}
            \item[$(i)$] $\mathcal{T}$ is a torsion class
            \item[$(ii)$] $\mathcal{F}$ is a torsion-free class
            \item[$(iii)$] $\Hom_A(\mathcal{T}, \mathcal{F})=0$
            \item[$(iv)$] For every $X \in \modu A$, there is a short exact sequence $0 \to tX \to X \to X/tX \to 0$, with $tX \in \mathcal{T}$ and $X/tX \in \mathcal{F}$.
        \end{itemize}
    \end{enumerate}
\end{defn}

Now, here are a few remarks.

\begin{remark}
    \begin{enumerate}
        \item[$(a)$] Conditions $(iii), (iv)$ above imply conditions $(i), (ii)$. Hence, the first two conditions in the definition of a torsion pair are not needed (but it is convenient to include them as part of the definition).
        \item[$(b)$] If $\mathcal{T}$ is a torsion class, then there is a unique torsion-free class $\mathcal{F}$ such that $(\mathcal{T}, \mathcal{F})$ is a torsion pair. We have 
        $$\mathcal{F} = \mathcal{T}^\perp:=\{X \in \modu A \mid \Hom(\mathcal{T},X)=0\}.$$
        \item[$(c)$] Dual to part $(b)$, if $\mathcal{F}$ is a torsion-free class, there is a unique torsion class $\mathcal{T}$ such that $(\mathcal{T}, \mathcal{F})$ is a torsion pair. In particular, $\mathcal{T} = ^\perp \mathcal{F}:=\{X \in \modu A \mid \Hom(X, \mathcal{F})=0\}$.
    \end{enumerate}
\end{remark}

Some subcategories of $\modu A$ are nicer in the sense that modules in $\modu A$ always admit approximations by objects in the subcategory.  To articulate this more precisely, we need the following definition.

\begin{defn}
    Let $\mathcal{C}$ be a subcategory of $\modu A$.
    \begin{enumerate}
        \item[$(1)$] The subcategory $\mathcal{C}$ is \emph{contravariantly finite} provided for every $X \in \modu A$, there exists a morphism $f: C \to X$ with $C \in \mathcal{C}$ such that $\Hom_A(\mathcal{C},f)$ is surjective. In other words, if $f': C' \to X$ is a morphism with $C' \in \mathcal{C}$, then there exists $h: C\to C'$ such that $f' = fh$:
        $$\xymatrix{C \ar[r]^f & X \\ C'\ar@{-->}[u]^{h} \ar[ur]_{f'} &}$$
        \item[$(2)$] The subcategory $\mathcal{C}$ is \emph{covariantly finite} provided for every $X \in \modu A$, there exists a morphism $f: X \to C$ with $C \in \mathcal{C}$ such that $\Hom_A(f,\mathcal{C})$ is surjective. Equivalently, if $f': X \to C'$ is a morphism with $C' \in \mathcal{C}$, then there exists $h: C'\to C$ such that $f' = hf$:
        $$\xymatrix{X \ar[r]^f \ar[dr]_{f'} & C \ar@{-->}[d]^{h} \\ & C'}$$
        \item[$(3)$] The category $\mathcal{C}$ is \emph{functorially finite} if it is both contravariantly finite and covariantly finite.
    \end{enumerate}
\end{defn}

\begin{remark}
    The morphism $f$ in (1) is called a \emph{right $\mathcal{C}$-approximation} of $X$. The morphism $f$ in (2) is called a \emph{left $\mathcal{C}$-approximation} of $X$. Hence, the subcategory $\mathcal{C}$ is functorially finite in $\modu A$ provided every $X \in \modu A$ admits a left and a right $\mathcal{C}$-approximation.
\end{remark}

The following result is a combination of some older results by Auslander and Smal\o  \cite{AS}, and some newer results by Adachi, Iyama and Reiten \cite{AIR}.

\begin{theorem}\label{Thm: funct. finite torsion pairs}
    Let $(\mathcal{T}, \mathcal{F})$ be a torsion pair. Then the following are equivalent
        \begin{enumerate}
            \item[$(1)$] $\mathcal{T}$ is functorially finite.
            \item[$(2)$] $\mathcal{T}$ is covariantly finite.
            \item[$(3)$] $\mathcal{T} = {\Fac}(M)$, for some $\tau$-rigid module $M$.
            \item[$(4)$] $\mathcal{T} = {\Fac}(T)$, for some support $\tau$-tilting module $T$.  \item[$(5)$] $\mathcal{F}$ is functorially finite.
            \item[$(6)$] $\mathcal{F}$ is contravariantly finite.
            \item[$(7)$] $\mathcal{F} = {\rm sub}(\tau T)$, for some $\tau$-rigid module $T$.
            \item[$(8)$] $\mathcal{F} = {\rm Sub}(\tau T)$, for some support $\tau$-tilting module $T$.
    \end{enumerate}
\end{theorem}

\begin{remark}
\begin{enumerate}
    \item[$(a)$] It is not hard to check that a module $T$ is $\tau$-rigid if and only if $\tau T$ is $\tau^{-1}$-rigid, where a module $Z$ is $\tau^{-1}$-rigid provided $\Hom(\tau^{-1}Z,Z)=0$.
    \item[$(b)$] Similar to the notion of (support) $\tau$-tilting module, we can define the $\tau^{-1}$ analogue, which is the notion of (support) $\tau^{-1}$-tilting module. It is also not hard to check that $T$ is support $\tau$-tilting if and only if $\tau T$ is support $\tau^{-1}$-tilting. Note that $T$ might be $\tau$-tilting without $\tau T$ being $\tau^{-1}$-tilting. This happens if and only if $T$ has a non-zero projective direct summand. In this case, however, $\tau T$ is still support $\tau^{-1}$-tilting.
    \item[$(c)$] With this, the equivalent conditions in the above theorem are also equivalent to
    \begin{enumerate}
    \item[$(9)$]  $\mathcal{F} = {\rm Sub}(Z)$, for some $\tau^{-1}$-rigid module $Z$.
        \item[$(10)$]  $\mathcal{F} = {\rm Sub}(Z)$, for some support $\tau^{-1}$-tilting module $Z$.
    \end{enumerate}
    \item[$(d)$] We note that the modules $M$ in part $(3)$, $T$ in part $(7)$, and $Z$ in part $(9)$ are not necessarily unique. In contrast, those modules in parts $(4)$, $(8)$, and $(10)$ are uniquely determined, up to isomorphism and basic form.
\end{enumerate}
    
\end{remark}

Now, let us turn our attention to bricks in the setting of torsion theory. We first mention the following.

\begin{problem}
    Show that a torion class in $\modu A$ is completely determined by the bricks it contains. In other words, show that if $\mathcal{T}_1, \mathcal{T}_2$ are two torsion classes in $\modu A$ that contain the same bricks, then $\mathcal{T}_1 = \mathcal{T}_2$. \emph{(Hint: Assume otherwise, and let $M$ be of minimal length that belongs to one torsion class and not the other. Then $M$ is not a brick. Consider the image of a non-zero non-isomorphism $f: M \to M$.) }
\end{problem}

Given a class $\mathcal{X}$ of modules in $\modu A$, we can consider $T(\mathcal{X})$ the smallest torsion class that contains $\mathcal{X}$. Alternatively, this can be expressed as the intersection of all torsion classes containing $\mathcal{X}$. Of course, we have the well known identity
$$T(\mathcal{X}) = {\Filt}({\Fac} \mathcal{X}).$$
A module $M$ lies in the latter if and only if it admits a filtration
$$0 = M_0 \subset M_1 \subset \cdots \subset M_{r-1} \subset M_r = M$$
such that for $i = 1, \ldots, r$ the module $M_i / M_{i-1}$ lies in ${\Fac} \mathcal{X}$. We recall that ${\Fac} \mathcal{X}$ consists of the modules which are quotient of a module in ${\rm add} \mathcal{X}$. If $\mathcal{X}$ is a single $\tau$-rigid module $T$, or its additive closure ${\rm add} T$, then we have the simpler expression
$$T(\mathcal{X}) = {\Fac} T.$$

The following proposition was proven by Demonet, Iyama and Jasso \cite{DIJ}. In \cite{As2}, Asai has shown an analogue and generalization of this correspondence for semibricks. Before we recall the ``brick-$\tau$-rigid" correspondence, we first introduce some notations. 

\medskip

Henceforth, we let $\brick(A)$ denote a complete set of non-isomorphic bricks in $\modu A$ and $i\taurigid(A)$ a complete set of non-isomorphic indecomposable $\tau$-rigid modules in $\modu A$. Moreover, we let $\tors(A)$ denote the set of all torsion classes in $\modu A$.

\begin{prop} \label{brick-tau-rigid corr.}
    There are injections
    $$\varphi: i\taurigid(A) \to \brick(A)$$
    and
    $$\psi: \brick(A) \to \tors(A)$$
    such that $\varphi(T) = \frac{T}{{\rm rad}_{\End(T)}T}$ and $\psi(B) = T(B)$. Moreover, $\psi\varphi(T) = {\Fac}(T)$.
\end{prop}

\begin{remark}
\begin{enumerate}

    \item[$(a)$] In the above formula, we remark that
    $${\rm rad}_{\End(T)}T = \sum_{f: T \to T, \text{non-iso}}{\rm Im}(f).$$
    
    \item[$(b)$] Observe that if $T$ is indecomposable $\tau$-rigid, then $\varphi(T)=T$ if and only if $T$ is a brick.
\end{enumerate}
   
\end{remark}

\begin{definition} The algebra $A$ is called
\begin{enumerate}
    \item[$(1)$] \emph{brick-finite} if the set $|\brick(A)|$ is finite.
    \item[$(2)$]  \emph{$\tau$-tilting finite} if the set $|i\taurigid(A)|$ is finite, or equivalently, if there are only finitely many support $\tau$-tilting modules, up to ismorphisms.
\end{enumerate}
     
\end{definition}
Now, we mention the following important result by Demonet, Iyama and Jasso which extends some earlier results on $\tau$-tilting finite algebras (for details, see \cite{DIJ}).

\begin{prop}\label{Prop: tau-tilting-finiteness}
The following conditions are equivalent.
\begin{enumerate}
    \item[$(1)$] $A$ is brick-finite
    \item[$(2)$] $A$ is $\tau$-tilting finite
    \item[$(3)$] $\modu(A)$ admits only finitely many torsion classes
    \item[$(4)$] all torsion classes of $\modu(A)$ are functorially finite. 
\end{enumerate}
\end{prop}

Note that Prop. \ref{brick-tau-rigid corr.}, together with the equivalence of (2) and (3), implies (1) and (2) are equivalent. 

\subsection{Brick labeling}
We end this section with some comments on the notion of brick labeling. The notion of brick labeling was first studied by Barnard, Carroll and Zhu in \cite{BCZ}. The idea is that any cover relation $\mathcal{U} \subset \mathcal{T}$ in the poset $\tors A$ (ordered by inclusion) can be associated with a unique element of $\brick (A)$. Their result was later developed by some other authors, including in \cite{DI+}. We start with the following fundamental result

\begin{prop}[\cite{DI+}]
    Let $\mathcal{U} \subset \mathcal{T}$ be a cover relation in the poset $\tors A$. Then there exists a unique bricks $B$ in $\mathcal{U}^\perp \cap \mathcal{T}$ (hence, $B$ is a unique brick that is torsion w.r.t the largest torsion class, and $B$ is torsion free w.r.t. the smallest torsion class).
\end{prop}

The brick in the above proposition is called the \emph{labeling brick} of the cover relation (equivalently, of the edge of the Hasse graph of $\tors A$). In \cite{BCZ}, given $\mathcal{T} \in \tors A$, the authors were able to homologically characterize all of the labeling bricks corresponding to cover relations $\mathcal{T} \subset \mathcal{U}$ for some $\mathcal{U}$. These are the \emph{minimal extending bricks} for $\mathcal{T}$. Similarly, the set of all labeling bricks corresponding to cover relations $\mathcal{U} \subset \mathcal{T}$, for some $\mathcal{U}$, are called the \emph{minimal co-extending bricks} for $\mathcal{T}$. Of note is the following result.

\begin{prop}\label{Prop: func-finite cover relations}
    Let $\mathcal{U} \subset \mathcal{T}$ be a cover relation in $\tors A$. Then $\mathcal{U}$ is functorially finite if and only if $\mathcal{T}$ is functorially finite.
\end{prop}

\begin{remark}
    The above proposition tells us that the given torsion classes of a connected component of the Hasse graph of $\tors A$ are either all functorially finite, or all non-functorially finite.
\end{remark}

Before we finish this section, we sketch how the labeling brick corresponding to a functorially finite cover relation $\mathcal{U} \subset \mathcal{T}$ can be computed.  Due to the preceding remark and Propositions \ref{Prop: tau-tilting-finiteness} and \ref{Prop: func-finite cover relations}, this gives all the labeling bricks of $\tors A$, for any $\tau$-tilting finite algebra $A$.
Since $\mathcal{U}$ is functorially finite, then $\mathcal{U} = \Fac U$ and $\mathcal{T} = \Fac T$ for some (basic) support $\tau$-tilting modules $U$ and $T$. Moreover, we have
$$U = Z \oplus U', \textit{and  \,} T = Z \oplus T'$$
where $T'$ is indecomposable and $U'$ is indecomposable or zero (in fact, $T$ is a right mutation of $U$, and $U=0$ when the support of $T$ is strictly larger than that of $U$). Now, the labeling brick $B$ is
$$B = \frac{T'}{\sum_{f: Z \to T'}{\rm Im f}}.$$

\vskip 0.5cm

\section{Geometry of bricks and an open conjecture}\label{Section: Geometry of bricks}

In this section, we first introduce some of the standard algebraic and geometric tools and known results that we freely use in our studies of bricks, before stating some recent results on the bricks and the \emph{second brick-Brauer-Thrall (2nd bBT) conjecture}.
Then, in the second part, we present a brief summary of the history and developments on this open conjecture, as well as some related results and important consequences of this conjecture. For the most part, we only provide references.

\subsection{Bricks and their orbits}

For an algebra $A=kQ/I$ of rank $n$, and a fixed dimension vector $\mathbf{d}:=(d_i)_{i=1}^n \in \mathbb{Z}^{Q_0}_{\geq 0}$, by $\rep(A,\mathbf{d})$ we denote the affine variety parametrizing the representations of $A$ (or of the bound quiver $(Q,I)$) of dimension vector $\mathbf{d}$. In particular, we have $$\rep(Q,\mathbf{d}):=\{(M_\alpha)_{\alpha \in Q_1} \,|\, M_{\alpha} \in {\rm {Mat}_{d_j\times d_i} (k)} \textit{, for } \alpha: i \rightarrow j \}.$$ 
Note that $\rep(A,\mathbf{d})$ is a closed subset of $\rep(Q,\mathbf{d})$ consisting of those representations of $Q$ satisfying the relations from $I$.

\begin{example}
Let $Q$ be the following quiver 
\[\xymatrix{1 \ar@/^/[r]^{\beta}& 2 \ar@/^/[l]^{\alpha} \ar@/^/[r]^{\gamma}& 3 \ar@/^/^{\delta}[l] }\]

and $I=\langle \beta\alpha-\delta\gamma, \alpha\beta, \gamma\delta\rangle$, and $\mathbf{d}=(2,2,2)$. 
Then, 
$$\rep(A,\mathbf{d}):=\{(M_\alpha, M_\beta, M_\gamma, M_\delta) \,|\, M_\alpha, M_\beta, M_\gamma, M_\delta \in {\rm Mat}_{2 \times 2}(k), M_\beta M_\alpha = M_\delta M_\gamma, M_\alpha M_\beta = M_\gamma M_\delta=0 \}.$$

\end{example}

\begin{remark}
If $Q$ is connected, then $\rep(A,\mathbf{d})$ is a connected variety. However, in general, $\rep(A,\mathbf{d})$ is not necessarily irreducible. Nonetheless, $\rep(A,\mathbf{d})$ always has only finitely many irreducible components.
\end{remark}

For a fixed dimension vector $\mathbf{d}$, the general linear group $\GL(\mathbf{d}):=\GL(d_1)\times \cdots \times\GL(d_n)$ acts on $\rep(A,\mathbf{d})$ via conjugation. More precisely, for $g=(g_i)_{i \in Q_0}$ in $\GL(\mathbf{d})$ and $M=(M_\alpha)_{\alpha \in Q_1}$ belonging to $\rep(A,\mathbf{d})$, we have $g M= (g_j M_{\alpha} g_i^{-1})_{\alpha: i \to j}$, where $\alpha$ runs through all arrows in $Q_1$. Observe that $\GL(\mathbf{d})$ acts on each one of the irreducible components of the variety $\rep(A,\mathbf{d})$.
For $M \in \rep(A,\mathbf{d})$, by $\mathcal{O}_M$ we denote the $\GL(\mathbf{d})$-orbit of $M$, and let $\overline{\mathcal{O}}_M$ denote the closure of $\mathcal{O}_M$, with respect to the Zariski topology. We remark that $\mathcal{O}_M$ is the isoclass of $M$ in $\rep(A,\mathbf{d})$, that is, $N \in \rep(A, \mathbf{d})$ lies in $\mathcal{O}_M$ if and only if $N$ is isomorphic to $M$.

\begin{remark}
For $M \in \rep(A,\mathbf{d})$, the orbit-stablizer theorem implies $$\dim \mathcal{O}_M= \dim \GL(\mathbf{d})-\dim \stab(M)$$
where $\stab(M)$ is thought of as an algebraic variety (in fact, it is an algebraic subgroup of $\GL(\mathbf{d})$). Since $\stab(M)$ can be identified with $\Aut(M)$, which we can think of as an open subset of the (irreducible) affine space $\End_A(M)$, we get that the dimension of the variety $\stab(M)$ coincides with the vector space dimension of $\End_A(M)$. Thus, we obtain $$\dim \mathcal{O}_M= \dim \GL(\mathbf{d})-\dim \mathrm{End}_A(M).$$ 
This particularly implies that the orbits of bricks always have the maximal dimension in the irreducible components of $\rep(A,\mathbf{d})$ that they belong to.
\end{remark}

\begin{prop}\label{Prop: orbit closure of bricks}
Let $\mathcal{Z}$ be an irreducible component of $\rep(A,\mathbf{d})$. If $\mathcal{Z}$ contains a brick, then exactly one of the following cases holds:
\begin{enumerate}
    \item $\mathcal{Z}=\overline{\mathcal{O}}_B$, for a brick $B$;
    \item $\mathcal{Z}=\overline{\mathcal{U}}$, where $\mathcal{U}$ is an infinite union of orbits of bricks.
\end{enumerate}
\end{prop}
\begin{proof}
The map $\psi:\mathcal{Z}\rightarrow \mathbb{Z}$ which sends $M$ to $\dim \mathcal{O}_M$ is known to be lower semi-continuous. This, together with the fact that orbits of bricks in $\mathcal{Z}$ are of maximal dimension yields that the collection of all bricks in $\mathcal{Z}$ forms an open set. Since $\mathcal{Z}$ is irreducible, this open set either consists of a single orbit, or infinitely many of them.
\end{proof}

The following proposition provides a dictionary between some algebraic and geometric properties of the modules under consideration.

\begin{prop}[Voigt, de la Peña \cite{dlP}] \label{Prop: vanishing Ext and orbit openennes}
Let $\mathcal{Z}$ be an irreducible component of $\rep(A,\mathbf{d})$ and $M \in \mathcal{Z}$. Then,
\begin{enumerate}
    \item If $\Ext^1(M,M)=0$, then $\mathcal{O}_M$ is open.
    \item If $\Ext^2(M,M)=0$ and $\mathcal{O}_M$ is open, then $\Ext^1(M,M)=0$.
\end{enumerate}

\end{prop}
Part $(1)$ of the preceding proposition is often known as the Voigt's Lemma. As shown in the first part of the following example, the converse of Voigt's Lemma does not hold in general. In other words, to have the converse of the Voigt's Lemma, the vanishing condition $\Ext^2(M,M)=0$ in part $(2)$ of the above proposition is necessary.

\begin{example}\label{Example: no converse of Voigt's lemma}
\begin{enumerate}
    \item Let $A=k[x]/\langle x^2 \rangle$. For the dimension vector $d=1$, the variety $\rep(A,\mathbf{d})$ is evidently irreducible, as it consists of only the orbit of the simple module $S$.
    In particular, $S$ is a brick and has open orbit. However, $\Ext^1(S,S)\neq 0$, because $S$ admits a non-split self-extension. We note that $\Ext^i(S,S)\neq 0$ for all $i$.

    \item Let $Q:2\doublerightarrow{\alpha}{\beta}1$ and consider the path algebra $A=kQ$. For the dimension vector $\mathbf{d}=(1,1)$, we have $\rep(A,\mathbf{d})=\{(x,y)\,|\, x, y \in $k$\}\simeq k^2$. In particular, observe that $(x,y)$ in $\rep(A,\mathbf{d})$ is a brick if and only if $(x,y) \neq (0,0)$, that is, $k^2\setminus \{(0,0)\}$ is the set of bricks in $\rep(A,\mathbf{d})$. Hence, viewed in $k^2\setminus \{(0,0)\}$, a typical orbit of a brick can be seen as a line that passes through the origin, minus the origin.    
\end{enumerate}
\end{example}

From the previous section, recall the notion of labeling bricks between functorially finite torsion classes. In general, for such a labeling brick $M$, we may have $\Ext^1(M,M) \neq 0$ and $\Ext^2(M,M)\neq 0$. Hence, even for $\tau$-tilting finite algebras, one cannot directly apply Proposition \ref{Prop: vanishing Ext and orbit openennes} to every such labeling brick to decide whether $M$ has an open orbit (For instance, see Example \ref{Example: no converse of Voigt's lemma}, part $(1)$.). However, we have the following result. In particular, over a $\tau$-tilting finite algebra, the assumption of the next theorem holds for every brick. More importantly, the theorem applies to all labeling bricks between functorially finite torsion classes over arbitrary algebras.

\begin{theorem}\cite[Theorem 6.1]{MP4}
Any labeling brick between functorially finite torsion classes has open orbit.
\end{theorem}

Using the above proposition, we get a geometric proof of the following theorem, which was first shown in \cite{ST}, using another approach (but for general fields). 
It is well known that $A$ is brick-infinite if and only if there are infinitely many non-isomorphic bricks that appear as labeling brick between functorially finite torsion classes. Moreover, as remarked before, over a brick-finite algebra, all bricks are of this form.

\begin{theorem}[\textbf{1st bBT}, {\cite{MP2}, \cite{ST}}]\label{Thm: 1st bBT}
If $A$ is brick-infinite (equivalently, $\tau$-tilting infinite), then there is no bound on the dimension of bricks.    
\end{theorem}

The above theorem is a modern counterpart of the \textbf{First Brauer-Thrall (1st BT) Theorem}: \emph{If $A$ is representation-infinite, then there is no bound on the dimension of indecomposables in $\modu A$}. 
Hence, the preceding theorem gives a brick-analogue of the 1st BT theorem. Consequently, we refer to the above theorem as the First brick-Brauer-Thrall theorem and abbreviate it by \textbf{1st bBT}.

\medskip

We further recall that the \textbf{Second Brauer-Thrall (2nd BT) Theorem} asserts a stronger version of 1st BT theorem, as follows: \emph{If $A$ is representation-infinite, then there is a strictly increasing sequence of integers $d_1<d_2<\cdots$ such that, for each $d_i$, there is an infinite family of indecomposables of dimension $d_i$}. For more details and historical remarks on the celebrated Brauer-Thrall conjectures -- now theorems -- see \cite{Ja}, \cite{Bo}, \cite{Ri3}, and references therein.

\medskip

Although the 1st bBT is a verbatim brick-analogue of 1st BT, that is, the role of indecomposables is replaced by bricks, we remark that a verbatim brick-analogue of the 2nd BT is known to be false (for instance, consider the path algebra of the Kronecker quiver). However, we can consider a conceptual brick-analogue of the 2nd BT conjecture. Due to some observations explained in the following remark, nowadays we refer to this open problem as the \textbf{Second brick-Brauer-Thrall (2nd bBT) Conjecture}.

\begin{conjecture}[\textbf{2nd bBT}, {\cite{Mo1}}]\label{2nd bBT Conj.}
If $A$ is brick-infinite (equivalently, $\tau$-tilting infinite), there is a dimension vector $\mathbf{d}$ such that $\rep(A,\mathbf{d})$ contains infinitely many orbits of bricks.
\end{conjecture}

In Subsection \ref{Subsec: History of 2nd bBT}, we provide some historical remarks on the original motivations and further insight into the above conjecture. Meanwhile, we remark that the converse of the assertion of Conjecture \ref{2nd bBT Conj.} is obviously true. 
For a more geometric formulation of this conjecture, see the paragraph following Question \ref{Ques: rigid bricks}, below. Moreover, in the next subsection we state a surprisingly elementary equivalence of the open implication of 2nd bBT conjecture in terms of the pairwise Hom-orthogonal modules. We recall that a family $\{X_i\}_{i\in I}$ in $\modu A$ is said to be \emph{pairwise Hom-orthogonal}, if $\Hom_A(X_i,X_j)=0=\Hom_A(X_j,X_i)$, for each distinct pair $i, j \in I$.

\subsection{Some historical remarks, new results, and instant implications of 2nd bBT Conjecture.}\label{Subsec: History of 2nd bBT}
\,

In 2019, Conjecture \ref{2nd bBT Conj.} was first posed by the first-named author -- as Conjecture 6.6 of the arXiv preprint of \cite{Mo3} -- originally phrased it in a slightly different language. This equivalent formulation of the 2nd bBT conjecture and some results on that were obtained during the PhD studies of the author and also appeared in his dissertation, defended in June 2020 (\cite[Chapter VI]{Mo1}). 
Those results were mainly derived from the decisive role of bricks in the $\tau$-tilting finiteness of minimal representation-infinite algebras (see \cite{Mo1}). In fact, the author posed the 2nd bBT conjecture to establish a novel connection between the notion of $\tau$-tilting finiteness and some recent important studies on the geometry of Schur representations (bricks) and their stability conditions. As discussed in \cite[Chapter VI]{Mo1}, this conjecture relates some foundational algebraic problems treated in \cite{DIJ} to some important geometric phenomena studied in \cite{CKW}. 
After an extensive search in the public literature and exchange of ideas with numerous experts in the field, to our knowledge the 2nd bBT conjecture (or any equivalent version of it) had not appeared in the press before its first appearance in the arXiv preprint of \cite{Mo3}, posted in October 2019.

\medskip

In 2020, and in the substantially revised version of the arXiv preprint of \cite{STV}, a problem equivalent to the 2nd bBT conjecture was posed, and it was called ``Second $\tau$-Brauer-Thrall conjecture". This problem was similarly motivated by the $\tau$-tilting finiteness of algebras, and it is a reformulation of the 2nd bBT conjecture (hence, equivalent to Conjecture 6.6 of the arXiv preprint of \cite{Mo3}). 
We note that the main result of \cite{STV} on the band brick modules was new and important, as it provided further evidence for the correctness of the 2nd bBT conjecture. More specifically, they settled the problem of $\tau$-tilting finiteness for the family of special biserial algebras. In retrospect, their results generalized some earlier studies on $\tau$-tilting finiteness of gentle algebras that appeared in 2018, later published in \cite{Pl2}, and those on string and minimal representation-infinite biserial algebras which appeared in 2019, later published in \cite{Mo2}.

\medskip

In 2021, and in the first joint work \cite{MP4} of the authors of this manuscript, we conducted a thorough study of a reductive argument on the 2nd bBT conjecture via ``minimal $\tau$-tilting infinite algebras" -- a new family introduced and partially treated in \cite[Chapter V.I]{Mo1}. As a consequence of some more general results in \cite{MP4}, later published in 2023, we gave an important reduction of the 2nd bBT conjecture. More specifically, we proved the following: \emph{To settle the 2nd bBT conjecture in full generality, it suffices to verify the assertion for those minimal brick-infinite algebras over which almost all bricks are faithful}. For details and the original formulation, see \cite[Theorem 1.4]{MP4}, where one can also find several important properties of the minimal $\tau$-tilting-infinite algebras. Furthermore, for some more applications of our reductive arguments in the study of the 2nd bBT conjecture, with special attention to the minimal brick-infinite algebras, see Section 8 and the references therein. 

\medskip

In 2022, in \cite{MP1}, we settled the 2nd bBT conjecture for all biserial algebras and proved a noticeably stronger result over this family. This was done via our reductive approach and based on our explicit classification of all minimal brick-infinite biserial algebras. Roughly speaking, for biserial algebras we reduced the 2nd bBT conjecture to a nice family of gentle algebras, called ``generalized barbell algebras", which were already treated in \cite[Chapter V.3]{Mo1}. In retrospect, we generalized and strengthened the results of \cite{Mo1, Mo2},  \cite{Pl2}, and \cite{STV}. In the same paper, we also treated ``generic bricks" over biserial algebras. We further posed a new problem, the so-called ``Generic-brick Conjecture", which is concerned with the behavior of infinite-dimensional bricks. Our results and this conjecture also closely relate to some important recent work on infinite-dimensional bricks announced in 2020 in \cite{Se}, published in 2023.

\medskip

In 2023, motivated by the 2nd bBT conjecture and the connections to some other challenging problems, in \cite{MP2} we conducted a study of geometric interactions between bricks and different notions of rigidity. We further treated the geometric counterparts of our algebraic problems, particularly in terms of brick components and $\tau$-regular components of representation varieties (for details, see Section \ref{Sect: Bricks and tau-regular components} and references therein). Around the same time, independent studies were conducted in \cite{Pf}, where the author obtained some interesting results on closely related problems. In \cite{MP2}, we obtained some new results on the 2nd bBT conjecture, and introduced a construction, the so-called ``$\tau$-convergence", which generates a $1$-parameter family of bricks over certain minimal brick-infinite algebras (for details, see \cite[Section 6 and 8]{MP2}). This provided further evidence for the correctness of the 2nd bBT conjecture. 

\medskip

In 2024, inspired by the growing applications of bricks in different areas of research, as well as a new series of problems closely related to our earlier work, we treated the 2nd bBT conjecture in a larger pool of open problem, which we generally refer to as the ``brick-Brauer-Thrall Conjectures". More precisely, in \cite{MP3}, we showed how these problems are conceptually related. In particular, we observed that some of the other conjectures are either equivalent to (or consequence of) the 2nd bBT conjecture, thus our earlier results directly applied to them and provided new horizons for future work. Moreover, in terms of pairwise Hom-orthogonal modules, we gave a significantly more elementary equivalence of the 2nd bBT conjecture: \emph{For a brick-infinite algebra $A$, the 2nd bBT conjecture holds if and only if $\modu A$ contains an infinite family of pairwise Hom-orthogonal modules of the same dimension}. For details, see \cite{MP3}.

\medskip

We finish the subsection by noting that the 2nd bBT conjecture has some strikingly interesting consequences (see Question \ref{Ques: rigid bricks}), as well as many applications in other directions of research. Some of the applications highlighted below should be known to experts, and some more recent linkages are discussed in the following sections. 
For instance, it is known that the correctness of the 2nd bBT conjecture gives an affirmative answer to the following question, which (to our knowledge) is still an open problem. Recall that $X \in \modu A$ is said to be \emph{rigid} if $\Ext^1_A(X,X)=0$.

\begin{question}[\textbf{Rigid bricks}]\label{Ques: rigid bricks}
If every $X\in \brick(A)$ is rigid, then is algebra $A$ necessarily brick-finite?
\end{question}

To put the above question into perspective, first note that the converse of the statement is not true in general, that is, there exists a brick-finite algebra $A$ such that $\Ext^1_A(X,X)\neq 0$, for some $X \in \brick(A)$ (e.g., let $A$ be a local algebra of dimension larger than one, and $X$ the simple module). 
On the other hand, observe that the 2nd bBT conjecture -- formulated in terms of orbits of bricks -- asserts the following statement: \emph{Every $X\in \brick(A)$ has open orbit $\Leftrightarrow$ $A$ is brick-finite}. 
Note that the implication $\Leftarrow$ is always true (see Proposition \ref{Prop: orbit closure of bricks}). Thus, if the 2nd bBT conjecture (i.e., the implication $\Rightarrow$) is affirmatively settled for $A$, then a positive answer to Question \ref{Ques: rigid bricks} follows immediately. This is thanks to the fact that any rigid module has an open orbit in the irreducible component containing it, by Proposition \ref{Prop: vanishing Ext and orbit openennes}.

\medskip

In fact, Question \ref{Ques: rigid bricks} closely relates to part $(1)$ of \cite[Conjecture 6.0.1]{Mo1}, which itself is a consequence of part$(2)$ of \cite[Conjecture 6.0.1]{Mo1}, namely the 2nd bBT conjecture. As remarked earlier, our interest in the study of the density and openness of orbits of bricks was inspired by some interesting results and new perspectives from \cite{CKW}. 
With regard to Question \ref{Ques: rigid bricks}, we also note that in \cite{MP2} we treated the behavior of bricks under a stronger notion of rigidity. More precisely, we strengthened some earlier results of \cite{Dr} and obtained the following classification result: \emph{If every $X\in \brick(A)$ is $\tau$-rigid, then $A$ is locally representation-directed, thus $A$ is representation-finite.} For more details on the 2nd bBT and the conceptual connections to some other open conjectures, see \cite{MP3} and the references therein.

\section{Bricks and stability conditions}

For an algebra $A$, we let $D^b(\modu A)$ denote the bounded derived category of $\modu A$. We also let $K^b(\proj A)$ denote the full triangulated subcategory consisting of the bounded complexes of projective modules. These two categories coincide if and only if $A$ has finite global dimension.
In general, the Grothendieck group $K_0(D^b(\modu A))$ of the bounded derived category is isomorphic to $\mathbb {Z}^n$. In fact, if $\{S_1, \ldots, S_n\}$ is the set of all (isomorphism classes) of simple modules in $\modu A$, then by choosing the standard basis $\{[S_1], \ldots, [S_n]\}$, for $M$ in $\modu A$, we identify the dimension vector of $M$, denoted by $\mathbf{d}_M$, with the corresponding element $[M]$ in $K_0(D^b(\modu A))$. 
Furthermore, we note that there is a natural identification of the Grothendieck groups $K_0(K^b({\rm proj}A))$ and $K_0(\rm proj A)$. In particular, we consider the basis $\{[P_1], \ldots, [P_n]\}$, given by the indecomposable projective modules and we obtain the isomorphism $K_0(\rm projA) \cong \mathbb{Z}^n$. Moreover, by $K_0(\rm proj A)_{\mathbb{Q}}$ and $K_0(\rm proj A)_{\mathbb{R}}$ we respectively denote $K_0(\rm proj A) \otimes_{\mathbb{Z}} \mathbb{Q}$ and $K_0(\rm proj A) \otimes_{\mathbb{Z}} \mathbb{R}$. Note that, $K_0(\rm proj A)_{\mathbb{Q}}\cong \mathbb{Q}^n$ and $K_0(\rm proj A)_{\mathbb{R}}\cong \mathbb{R}^n$.

\medskip

Before we recall the important notion of (semi)stability, note that, for $\theta$ in $K_0(\proj A)_{\mathbb{R}}$ written in the above basis, and for each $M \in \modu A$, we define $\theta([M]):=\theta \cdot \mathbf{d}_M$.

\begin{definition}[\cite{Ki}]
For $\theta$ in $K_0(\proj A)$, and $M \in \modu A$, we say that
\begin{enumerate}
    \item $M$ is \emph{$\theta$-semistable} if $\theta([M])=0$, and for every submodule $L$ of $M$, we have $\theta([L])\leq 0$. 
    
    \item $M$ is \emph{$\theta$-stable} if $\theta([M])=0$, and for every nonzero proper submodule $L$ of $M$, we have the strict inequality $\theta([L])< 0$. 
\end{enumerate}
\end{definition}

\begin{example}
Let $Q$ be the quiver 
$$
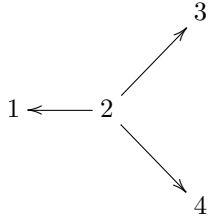
\xymatrix{&&3 \\ 1 & 2 \ar[l] \ar[ur] \ar[dr] &\\ &&4}$$ and consider the representation $M$ given by
$$
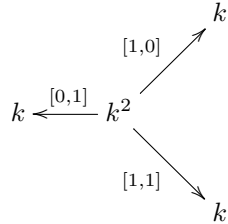
\xymatrix{&&k \\ k & k^2 \ar[l]_{[0,1]} \ar[ur]^{[1,0]} \ar[dr]_{[1,1]} &\\ &&k}$$
We have $\mathbf{d}_M=(1,2,1,1)$. Then, for $\theta=(-2,3,-2,-2)$, we can easily check that $M$ is $\theta$-stable, using that the only proper subrepresentations of $M$ which are non-zero at $2$ have dimension vectors $(1,1,0,0), (0,1,1,0)$ and $(0,1,0,1)$.
\end{example}

Before we highlight some connections between bricks, stable modules, and some subcategories of $\modu A$, we need to recall some standard terminology.
In particular, recall that a full subcategory $\mathcal{C}$ of $\modu A$ is said to be a \emph{wide} subcategory if it is closed under extensions, kernels, and cokernels. In the following, by $\wide(A)$ we denote the set of all wide subcategories of $\modu A$. 

For an extension-closed subcategory $\mathcal{E}$ of $\modu A$, recall that an object $Y$ in $\mathcal{E}$ in called \emph{simple} if it cannot be written as a non-split extension of the form $0\rightarrow X \rightarrow Y \rightarrow Z \rightarrow 0$, where $X$ and $Z$ are nonzero objects in $\mathcal{E}$. Let $\simp(\mathcal{E})$ denote the set of all simple objects in $\mathcal{E}$.

An easy argument shows that for each $\mathcal{W} \in \wide(A)$, every simple object of $\mathcal{W}$ must be a brick. In the following, we recall an elegant bijection between wide subcategories and some particular sets of bricks. Before stating the next proposition, we recall a terminology and, following Asai \cite{As2}, we say that a subset of bricks in $\modu A$ is a \emph{semibrick} if it consists of pairwise Hom-orthogonal bricks. Namely, a subset $\mathcal{S} \subseteq \brick(A)$ is a semibrick if, for each pair of distinct elements $X$ and $Y$ in $\mathcal{S}$, we have $\Hom_A(X,Y)=0=\Hom_A(Y,X)$. We remark that a semibrick can be infinite, that is, $\mathcal{S}$ can be an infinite subset of pairwise Hom-orthogonal elements of $\brick(A)$.
Observe that bricks (respectively, semibricks) can be viewed as a conceptual generalization of simple (respectively, semisimple) modules. Henceforth, by $\s-brick(A)$ we denote the set of all semibricks in $\modu A$.

\begin{prop}[\cite{Ri2}]\label{Prop: Ringle-wide-semibrick}
There is a bijection between $\wide(A)$
 and $\s-brick(A)$. More precisely, for each $\mathcal{W} \in \wide(A)$, we have that $\simp(\mathcal{W})$ is a semibrick. And, for each semibrick $\mathcal{S} \in \s-brick(A)$, we have that $\Filt(\mathcal{S})$ is a wide subcategory. Moreover, $\simp(\Filt(\mathcal{S}))=\mathcal{S}$ and $\Filt(\simp(\mathcal{W}))=\mathcal{W}$.
\end{prop}

As further shown recently, the idea of filtration by bricks plays a decisive role in the study of some fundamental problems in representation theory, particularly in connection with the maximal green sequences and lattice theory of torsion classes. More specifically, following \cite[Appendix]{KD}, in a recent survey by Ringel \cite{Ri1}, the notion of brick-filtration is studied in a more general setting. In particular, it is shown that any module has brick chain filtrations (see \cite[Theorem 1.2]{Ri1}). Moreover, the earlier work of Demonet in \cite{KD} has inspired new studies of maximal chains of bricks, which is further developed in \cite{AI+}, where the authors introduced the notions of brick-splitting torsion pairs and brick-directed algebras, respectively as the modern analogue and generalizations of the classical splitting torsion pairs and representation-directed algebras (for the definitions and more details, see \cite{AI+}).

\medskip

We also recall the following result from \cite{MS}. This, together with the preceding proposition, results in an injective map from $\s-brick(A)$ to $\tors(A)$.

\begin{proposition}\label{Prop: MS injective}
For an algebra $A$, there is an injective map $\wide(A)$ to $\tors(A)$, which sends $\mathcal{W}$ to $\Filt(\Fac(\mathcal{W}))$.
\end{proposition}

The following result establishes an interesting connection between the notion of (semi)stability and certain wide subcategories.

\begin{prop}[\cite{Ki}]
For $\theta$ in $K_0(\proj A)$, let $\mathcal{W}_{\theta}:=\{M \in  \modu A\,|\, M \text{ is $\theta$-semistable} \}$. Then,
\begin{enumerate}
    \item $\mathcal{W}_{\theta}$ is a wide subcategory.
    \item The simple objects in $\mathcal{W}_{\theta}$ are exactly the $\theta$-stable objects.
    \item Every $\theta$-stable module is a brick.
\end{enumerate}
\end{prop}

\begin{example}
For $Q:2\doublerightarrow{\alpha}{\beta}1$, consider the path algebra $A=kQ$, and let $\theta=(1,-1)$.
From the definition, it follows that if $M \in \modu A$ is $\theta$-semistable, then $\theta.\mathbf{d}_M=0$, therefore $\mathbf{d}_M=(-d,d)$. This implies that if $M$ is indecomposable, then it is a regular representation of the Kronecker quiver. Consequently, $\mathcal{W}_\theta$ is a subcategory of the category of all regular representations of $Q$. It is not hard to check that the indecomposable regular representations of dimension vector $(1,1)$ are all $\theta$-stable. Since $\mathcal{W}_\theta$ is closed under extensions, we see then that $\mathcal{W}_\theta$ coincides with the subcategory of $\modu A$ consisting of all regular representations. In particular, the $\theta$-stable modules are exactly the quasi-simple regular representations.
\end{example}

\section{Bricks and $\tau$-regular components}\label{Sect: Bricks and tau-regular components}

For an algebra $A$ of rank $n$, and a fixed dimension vector $\mathbf{d}$, by $\Irr(A,\mathbf{d})$ we denote the set of all irreducible components of the representation variety $\rep(A,\mathbf{d})$. Moreover, we set $\Irr(A)=\bigcup_{\mathbf{d}\in \mathbb{Z}_{\geq 0}^n}\Irr(A,\mathbf{d})$.

For each $\mathcal{Z} \in \Irr(A)$, define 
$c(\Z):=\min \{\dim(\Z) - \dim(\mathcal{O}_Z) \,| Z \in \Z \}$. Observe that $c(\Z)$ is the generic number of parameters. Moreover, we put $$e(\Z):=\min \{\dim \Ext^1_A(Z,Z) \,| Z \in \Z \} \hskip 0.5cm \text{and} \hskip 0.5cm h(\Z):=\min \{\dim \Hom_A(Z,\tau_A Z) \,| Z \in \Z \}.$$
We observe that $c(\mathcal{Z})$, $e(\mathcal{Z})$ and $h(\mathcal{Z})$ are attained on a non-empty open subset of $\mathcal{Z}$, and we always have the inequalities 
$c(\mathcal{Z})\leq e(\mathcal{Z})\leq h(\mathcal{Z})$. The following notion first appeared in \cite{GLS} under the name of strongly reduced components. In some later works, such components appeared under the name of (generically) $\tau$-reduced components. However, to avoid further confusion, henceforth we adopt the suitably justified terminology recently proposed in \cite{BS} and call these components (generically) $\tau$-regular.

\medskip

\begin{definition}
An irreducible component $\mathcal{Z}\in \Irr(A)$ is called \emph{(generically) $\tau$-regular} provided that $c(\mathcal{Z})= e(\mathcal{Z})= h(\mathcal{Z})$.
\end{definition}

\begin{example}
\begin{enumerate}
    \item For an acyclic quiver $Q$ and $A=kQ$, every $\mathcal{Z} \in \Irr(A)$ is $\tau$-regular.
    \item Let $A$ be an arbitrary algebra and $\mathcal{Z}\in \Irr(A)$. If $\mathcal{Z}$ contains a module $M$ with $\pd_A(M)\leq 1$, then $\mathcal{Z}$ is $\tau$-regular. To conclude this fact, first note that the set of modules of projective dimension at most $1$ forms an open set in $\mathcal{Z}$. Moreover, for a module $M$ in such an open set, one has $\dim(\Z) - \dim(\mathcal{O}_M)=\dim \Ext^1_A(M,M)=\dim \Hom_A(M,\tau M)$. 
\end{enumerate}
\end{example}

Below, we recall an important characterization of $\tau$-regular components. Before we state this result, we need to recall some terminology and introduce some notations.

As before, by considering the basis $\{[P_1], \ldots, [P_n]\}$ given by the indecomposable projective $A$-modules, we have the isomorphism $K_0(\rm projA) \cong \mathbb{Z}^n$.
We recall that every element of $K_0(\rm proj A)$ is often called a \emph{$g$-vector} of $A$. 
For each $g$-vector, say $\theta$, there is a canonical way to write $\theta$ as a difference of two vectors with non-negative entries. More precisely, let $\theta$ = $\theta^+ - \theta^-$, where $\theta^+ = (\theta'_i)$ such that $\theta_i' := {\rm max}\{0, \theta_i\}$, and $\theta^-=(\theta''_i)$ given by $\theta''_i := -{\rm min}\{0, \theta_i\}$. Let $P(\theta^+) := \oplus_{i=1}^n P_i^{(\theta^+)_i}$ and $P(\theta^-) := \oplus_{i=1}^n P_i^{(\theta^-)_i}$. Then, to $\theta$, we can associate the vector space $\Hom_A(P(\theta^-), P(\theta^+))$, where each element represents a $2$-term projective complex $P(\theta^-) \to P(\theta^+)$ in $K^b({\proj}(A))$. 

\medskip

Note that, for each $g$-vector $\theta$, we can view $\Hom_A(P(\theta^-), P(\theta^+))$ as an affine variety, consisting of points $f = (f_i)$ in the vector space $\prod_{i=1}^n\Hom_k(P(\theta^-)_i, P(\theta^+)_i)$, where the linear conditions $f_jP(\theta^-)(\alpha) = P(\theta^+)(\alpha)f_i$ are simultaneously satisfied, for all $\alpha: i \to j$ in $Q_1$. Alternatively, every $f = (f_i)$ in $\Hom_A(P(\theta^-), P(\theta^+))$ can be seen as a matrix of size $|\theta^+| \times |\theta^-|$, where for a vector $u \in \mathbb{Z}^n$, by $|u|$ we denote the sum of its entries. Each entry of this matrix associated to $f = (f_i)$ represents the maps between some indecomposable projective $A$-modules. In particular, if the entry corresponds to a map $P_i \rightarrow P_j$, there is a basis
$B_{ij}=\{p_{ij}^1, \ldots, p_{ij}^{s_{ij}}\}$ of $\Hom_A(P_i, P_j) = e_iAe_j$, and every entry corresponding to $P_i\to P_j$ is a linear combinations of the elements in $B_{ij}$.

\medskip

The following theorem was proved by Plamondon and gives an elegant description of $\tau$-regular components. For more details and the proof, see \cite{Pl1}. 
To state the following result more succinctly, we call a $g$-vector $\theta$ \emph{sincere} if for an element $f$ in general position in $\Hom_A(P(\theta^-), P(\theta^+))$, we have ${\rm ker}f \subseteq {\rm rad} P(\theta^-)$.

\begin{theorem}\label{Plamondon's result}
For algebra $A$, there is a bijection between the set of all $\tau$-regular components in $\Irr(A)$ and the set of all sincere $g$-vectors $\theta \in K_0(\proj A)$. 

More precisely, given $\theta$ with $\theta$ = $\theta^+ - \theta^-$, there is a nonempty open subset $O$ in  $\Hom_A(P(\theta^-), P(\theta^+))$ such that $\overline{\bigcup_{f \in O} \coker f}$ is a $\tau$-regular component in $\Irr(A)$. Moreover, every $\tau$-regular component can be obtained this way.
\end{theorem}

\begin{remark}
In the preceding theorem, the condition that $\theta$ is sincere is equivalent to the condition that no copy of $-[P_i]$, for some $i$, appears in the canonical decomposition of $\theta$.
\end{remark}

\section{Semi-invariants and stably-discrete algebras}\label{Sec: Semi-invatiant and CKW conj.}

As before, for an algebra $A$ and a dimension vector $\mathbf{d}$, if $\mathcal{Z}\in \Irr(A,\mathbf{d})$, then we consider the action of the general linear group $\GL(\mathbf{d})$ on $\Z$.
This naturally induces the action of $\GL(\mathbf{d})$ on the coordinate ring $\mathcal{Z}$, given by $g.f(z)=f(g^{-1}.z)$, where $g \in \GL(\mathbf{d})$ and $f \in k[\Z]$. Consequently, one can ask about the invariants of this action, the set denoted by $k[\Z]^{\GL(\mathbf{d})}$, consisting of all $f \in k[\Z]$ for which $g.f=f$ for all $g$ in $\GL(\mathbf{d})$. 
One can simply check that $k[\Z]^{\GL(\mathbf{d})}=k$ is the trivial subring of $k[\Z]$. This is because for every representation $M$ in $\mathcal{Z}$ and every oriented cycle $p$ in $Q$, the admissibility of ideal $I$ implies that $M(p)$ is nilpotent. For the rudiments of (semi-)invariant theory in the context of this manuscript, we refer to \cite{DW}.

\medskip

In the following, we primarily consider the ring of semi-invariants, denoted by $\SI(\Z):=k[\Z]^{\SL(\mathbf{d})}$, where the special linear group $\SL(\mathbf{d}):=\SL(d_1)\times \cdots \times\SL(d_n)$ acts on $k[\Z]$, as specified above.
For every $g$-vector $\theta=(\theta_i)\in K_0(\proj A)\simeq \mathbb{Z}^n$, define $\mathcal{X}_{\theta}: \GL(\mathbf{d})\rightarrow k^*$ that sends $g=(g_i)_{i\in Q_0)}$ to $\prod_{i\in Q_0} (\det g_i)^{\theta_i}$.
We note that $\mathcal{X}_\theta$ is a multiplicative character and all such characters can be described like this. We set 
$$\SI(\Z)_{\theta}:=\{f \in k[\Z] \, | \, g.f=\mathcal{X}_\theta(g)f, \text{ for all }  g \in \GL(\mathbf{d}) \},$$
and consider that weight decomposition $\SI(\Z)=\bigoplus_{\theta \in K_0(\proj A)} \SI(\Z)_{\theta}$.

\subsection{Schofield semi-invariants}

Let $\theta =[P_0]-[P_{-1}] \in K_0(\proj A)$, and for each $f$ in $\Hom(P_{-1},P_0)$, set $V_f:= \coker f$. We note that $f$ is the minimal projective presentation of $V_f$ if and only if ${\rm ker}f \subseteq {\rm rad} P(\theta^-)$. If $f$ is in general position, then $f$ is the minimal projective presentation of $V_f$ whenever $\theta$ is sincere.

\medskip

In all cases, we let $P$ be the maximal direct summand of $P(\theta^-)$ that belongs to the kernel of $f$.
Then, for $M \in \modu A$, if we apply the functor $\Hom(-,M)$ to the exact sequence $ P_{-1}\xrightarrow{f} P_0 \rightarrow V_f \rightarrow 0$, we obtain the following exact sequence
$$0 \to \Hom(V_f,M) \to \Hom(P_0,M) \xrightarrow{f_M} \Hom(P_1,M) \to \D\Hom(M, \tau V_f) \oplus \Hom(P,M) \to 0.$$
Now, observe that 
\begin{enumerate}
    \item $f_M$ is a square matrix if and only if $\dim \Hom(P_0,M) - \dim \Hom(P_1,M)=0$. This is the case if and only if $\theta \cdot \mathbf{d}_M=0$, namely, $\theta$ is orthogonal to the dimension vector of $M$.

    \item $f_M$ is an invertible matrix if and only if $\Hom(V_f,M)=0=\Hom(M,\tau V_f) = \Hom(P,M)$. This is the case if and only if $M \in \big(V_f^{\perp} \cap \,^{\perp}\tau V_f \cap P^\perp \big)$.
\end{enumerate}

With the same notation as above, set $C^f:= \det \Hom(f,-)$ and observe that it is defined on the set of all $M \in \modu A$ with $\theta  \cdot \mathbf{d}_M=0$. Then, we have the following important theorem. We note that it was first proven in the quiver case by Schofield and Van den Bergh \cite{SVdB} in the characteristic zero case and independently by Derksen and Weyman \cite{DW2} in arbitrary characteristic. Then, it was extended by Derksen and Weyman in \cite{DW3} in the case of a triangular algebra, and then by Domokos \cite{Do} for general algebras, but in the characteristic zero case.

\begin{theorem}[\cite{Do}]
Let the characteristic of $k$ be zero. With the same notation as above, each $C^f$ is a semi-invariant of weight $\theta \in K_0(\proj A)$. Furthermore, $\SI(\Z)_{\theta}$ is spanned, over $k$, by all such $C^f$.
\end{theorem}

The following proposition is useful in the study of semistability of modules in an irreducible component. This fundamental result explains why semistability is related to semi-invariants.

\begin{prop}[\cite{Ki}]
Let $\mathcal{Z}\in \Irr(A,\mathbf{d})$. Then, $M \in \mathcal{Z}$ is $\theta$-semistable if and only if there exist $n \in \mathbb{Z}_{\geq 0}$ and some $f \in \SI(\Z)_{n{\theta}}$ such that $f(M)\neq 0$.
\end{prop}

The following remark establishes a connection between semi-invariants and (generalized) $\tau$-perpendicular categories. Indeed, notice that in the remark below, if $\mathcal{Z}_\theta$ has $c(\mathcal{Z}_\theta)=0$, then $V_f$ is $\tau$-rigid.

\begin{remark}
Let $\theta =[P_0]-[P_{-1}] \in K_0(\proj A)$, with the generic decomposition $(R\rightarrow 0)\oplus (P'_{-1}\rightarrow P_0)$, where $(R\rightarrow 0)$ is the maximal shifted projective summand. For $f \in \Hom(P_{-1},P_0)$ in general position, we have $C^f(M)\neq 0$ if and only if $\Hom(V_f,M)=\Hom(M,\tau V_f)=\Hom(R,M)=0$, where $V_f$ is in general position of the $\tau$-regular component $\Z_{\theta}$. 
In other words, $M$ is $\theta$-semistable if and only if there exists $f$ in $\Hom(P_{-1},P_0)$ for which $M \in \big( R^{\perp} \cap V_f^{\perp} \cap \,^{\perp}\tau V_f \big)$, with $V_f$ in the general position of the $\tau$-regular component associated to $\theta$ (or formally, to $\theta - [R]$).
\end{remark}

\begin{definition}[\cite{Ki}]
Let $\mathcal{Z}\in \Irr(A,\mathbf{d})$ and $\theta \in K_0(\proj A)$. Then, define the projective variety
$$\mathcal{M}^{\theta-ss}(\mathcal{Z}):=\proj \Big( \bigoplus_{n\in \mathbb{Z}_{\geq 0}}\SI(\Z)_{n\theta} \Big),$$ which is called the \emph{moduli space of $\theta$-semistable representations}.
\end{definition}

The following proposition describes some fo the fundamental properties of this projective variety.

\begin{prop}[\cite{Ki}]
Points in $\mathcal{M}^{\theta-ss}(\mathcal{Z})$ are the $\theta$-semistable representations in $\mathcal{Z}$, considered up to the following equivalent relation: $M\sim_{ss} N$ if and only if $M$ and $N$ have the same $\theta$-stable composition factors.
\end{prop}

\begin{remark}
\begin{enumerate}
    \item Let $\mathcal{M}^{\theta-s}(\mathcal{Z})$ denote the set of all $\theta$-stable points in $\mathcal{M}^{\theta-ss}(\mathcal{Z})$. Then, $\mathcal{M}^{\theta-s}(\mathcal{Z})$ is an open (possibly empty) subset of $\mathcal{M}^{\theta-ss}(\mathcal{Z})$.

    \item For all $\mathcal{Z}$ and all $\theta$, we have that $\mathcal{M}^{\theta-ss} (\mathcal{Z})$ are zero-dimensional if and only if the dimension of $\SI(\Z)_{\theta}$ is $0$ or $1$. This is the case if and only if, for every $\theta$, all $\theta$-stable representations have open orbits.
\end{enumerate}
\end{remark}

We remark that the following notion first appeared in \cite{CKW} and the authors treated it under the name of \emph{multiplicity-free} algebras. However, below we propose a new terminology and call such algebras \emph{stably-discrete}, as we believe it is more enlightening in our context. More specifically, we observe that this geometric notion implies that an algebra which satisfies the condition of the following definition has no infinite family consisting of stable modules in any of its irreducible components, that is, for each $\mathcal{Z} \in \Irr(A)$, there is at most one stable module (also compare with Definition \ref{Def: brick-discrete} and Remark \ref{Rem: brick-discrete and stably-discrete}).

\begin{definition}
An algebra $A$ is said to be \emph{stably-discrete} if, for each irreducible component $\mathcal{Z} \in \Irr(A)$ and every $\theta \in K_0(\proj A)$, the moduli space $\mathcal{M}^{\theta-ss}(\mathcal{Z})$ is zero-dimensional.
\end{definition}

Following \cite{MP3}, we also recall the following definition.

\begin{definition}\label{Def: brick-discrete}
An algebra $A$ is called \emph{brick-discrete} if every brick in $\brick(A)$ has open orbit.
\end{definition}

We note that the above terminology was inspired by the geometric properties of bricks. In particular, $A$ is brick-discrete if and only if for each $\mathcal{Z} \in \Irr(A)$, there is at most one (orbit) of bricks $\mathcal{Z}$. 
We further observe that, in light of the 2nd bBT conjecture and the generic-brick conjecture posed by the authors, the notion of brick-discrete algebras is expected to also coincide with the discreteness of the spectrum associated to arbitrary finite dimensional algebras, as studied by Ringel in \cite{Ri4}. More specifically, for any algebra $A$, each $B \in \brick(A)$ can be seen as a ring empimorphism $\delta_B: A\rightarrow \Mat_d(k)$, where $\dim_k(B)=d$. The set of such ring epimorphisms $\delta: A\rightarrow \Mat_d(D)$ from $A$ to simple Artinian rings generalize the notion of spectrum from commutative algebras to arbitrary (finite dimensional) algebras. Consequently, it provides another perspective to the study of bricks, particularly in connection with the decisive role of generic bricks. For more details on this perspective and some new results on generic bricks over tame algebras, we refer to \cite{MP1}, \cite{BPS}, and references therein.

\begin{remark}\label{Rem: brick-discrete and stably-discrete}
To put the above definitions in contrast, we note that
    \begin{enumerate}
    \item Every brick-finite algebra is brick-discrete. Moreover, the converse is equivalent to the assertion of the Second brick-Brauer-Thrall conjecture (see Conjecture \ref{2nd bBT Conj.}). More precisely, the $2nd$ bBT conjecture states that an algebra $A$ is brick-discrete if and only if $A$ is brick-finite.
    \item If $A$ is brick-discrete, then $A$ is stably-discrete. Moreover, as explained below, the converse is expected to be true (see Theorem \ref{Thm: CKW on tame} and Conjecture \ref{CKW Conj.}).
    \item If $A$ is not brick-discrete, then it follows from \cite[Proposition 3.4]{MP3}, which heavily uses the results from \cite{GLFS}, that there exists and infinite semibrick consisting of bricks of the same dimension. In particular, in this case, $A$ satisfies a stronger version of the semibrick conjecture, originally posed in \cite{En} (for details, see \cite[Remark 3.5]{MP3}).
    \end{enumerate}
\end{remark}

For tame algebras, the following theorem establishes further connections between the new notions introduced above.  The proof of this theorem, which originally appeared in \cite{CKW}, provides further insight into the problems that we consider below. Hence, we present a proof of this interesting result.

\begin{theorem}[\cite{CKW}]\label{Thm: CKW on tame}
Let $A$ be a tame algebra. Then, $A$ is stably-discrete if and only if $A$ is brick-discrete.
\end{theorem}
\begin{proof}
We only need to show that if $A$ is stably-discrete, then $A$ is brick-discrete. If we assume otherwise, there exists an irreducible component $\Z \in \Irr(A)$ with $\Z=\overline{\mathcal{U}}$, where $\mathcal{U}$ is an open subset of $\mathcal{Z}$ which consists of an infinite union of orbits of bricks.

Since $A$ is tame, then $c(\mathcal{Z})=1$. Moreover, for some nonempty open subset $\mathcal{O} \subseteq \mathcal{U}$, and each $B \in \mathcal{O}$, we have $\tau B\simeq B$. This implies that $h(\Z)=1$, and therefore $\Z$ is a $\tau$-regular component.
Hence, by Theorem \ref{Plamondon's result}, there exists a $g$-vector $\theta_{\Z} \in K_0(\proj A)$ corresponding to $\Z$. We claim that elements of $\mathcal{O}$ are $\theta_{\Z}$-stable, and from that we conclude the desired contradiction.

To show the aforementioned claim, first observe that for any choice of $B$ in $\mathcal{O}$, due to the description of $\theta_{\Z} \in K_0(\proj A)$, we have $\theta_{\Z}(-)=\dim \Hom_A(B,-)-\dim \Hom_A(-,\tau B)$. This evidently implies $\theta_{\Z}([B])=0$. Moreover, if $L$ is a nonzero proper submodule of $B$, then $\Hom_A(L,\tau B)=\Hom_A(L,B)\neq 0$, therefore $\dim \Hom_A(L,\tau B)\neq 0$. On the other hand, since $B$ is a brick, we have $\Hom_A(B,L)=0$. This implies $\theta_{\Z}([L])<0$, and therefore each $B$ in $\mathcal{O}$ is in fact $\theta_{\Z}$-stable. However, these $\theta_{\Z}$-stable modules do not have open orbits, which contradicts the assumption that $A$ is stably-discrete.
\end{proof}

As articulated below, in \cite{CKW} the authors conjectured that the statement of the above theorem holds for arbitrary algebras. Meanwhile, let us remark that the following conjecture was originally phrased in a slightly different (but equivalent) language. For more details on the new results on this conjecture, as well as the connections between this problem and some other open conjectures, we refer to \cite{MP3} and references therein (also see Theorem \ref{Thm: equivalences for tame algebras}).

\begin{conjecture}\cite{CKW}\label{CKW Conj.}
An algebra $A$ is stably-discrete if and only if $A$ is brick-discrete.
\end{conjecture}

\section{Brick-finiteness and brick-discreteness}

Here we recall some of the more recent notions of finiteness and tameness that have been introduced in the past few years and are motivated by different phenomena in the study of representation theory of algebras. Then, we treat these notions through the lens of some open conjectures, particularly from the viewpoint of the behavior of bricks in the algebraic, geometric and homological settings.

\medskip

We first recall the required terminology and notation. As before, $A$ denotes an algebra of rank $n$, and a basic $\tau$-rigid pair $(M,P)$ consists of a basic $\tau$-rigid module $M$, and a basic projective module $P$ in $\modu A$ such that $\Hom_A(P,M)=0$.  
For any such basic $\tau$-rigid pair, it is known that the $g$-vectors of the indecomposable summands of $M$ and $P$ are linearly independent in $K_0(\rm proj A)_{\mathbb{R}} = \mathbb{R}^n$. In particular, these $g$-vectors specify the rays of a polyhedral cone $C(M, P)$ in $\mathbb{R}^n$ which is associated to $(M,P)$. 
We remark that $C(M, P)$ is a cone of maximal dimension if and only if $(M,P)$ is a support $\tau$-tilting pair, namely, $|M|+|P|=n$. 
Below, we recall an important theorem which uses some fundamental facts in $\tau$-tilting theory to associate a fan to the algebra under consideration. Due to the conceptual connection between the construction of this fan and $\tau$-tilting theory, henceforth, we call the fan described in the theorem as \emph{$\tau$-tilting fan} of $A$ and denote it by $\mathcal{F}_A$. 
For more details on the construction of this fan and its properties, we refer to \cite{AIR} and \cite{DIJ}, and references therein.

\medskip

\begin{remark}
The fan described in the following theorem is often viewed as a conceptual counterpart of the $g$-vector fans of cluster algebras. Consequently, in the literature this fan is also studied under the name ``$g$-vector fan" of $A$, and sometimes ``silting fan". However, we prefer to simply call this fan the \emph{(support) $\tau$-tilting fan} of $A$. In doing so, we have two conceptual observations: On the one hand, the term ``$\tau$-tilting fan" better indicates the connections between the construction of the fan and the $\tau$-tilting theory of the algebra under consideration. On the other hand, and even more importantly, we emphasize that every $\tau$-tilting infinite algebra $A$ has some element of $K_0(\rm proj A)_{\mathbb{R}}$ (potentially even in $K_0(\rm proj A)$) which do not belong to the $\tau$-tilting fan $\mathcal{F}$; see Proposition \ref{Prop: g-finite equ. brick-finite}.
\end{remark}

\begin{theorem}[\cite{AIR}] \label{Thm: tau-fan}
For an algebra $A$ of rank $n$, the rays
of the fan $\mathcal{F}_A$ in $\mathbb{R}^n$ are the $g$-vectors of the indecomposable $\tau$-rigid modules in $\modu A$ together with the $g$-vectors of the shifts of indecomposable projectives. In particular, the rays of each maximal cone of $\mathcal{F}_A$ are given by the $g$-vectors of the indecomposable summands of a given support $\tau$-tilting pair $(M,P)$. Moreover, the fan $\mathcal{F}_A$ is essential, rational, polyhedral, and simplicial.
\end{theorem}

When there is no risk of confusion, by $\mathcal{F}_A$, or simply $\mathcal{F}$  we denote the $\tau$-tilting fan of $A$ given in the above theorem. Unless specified otherwise, we always consider $\mathcal{F}_A$ in $K_0(\proj(A))_\mathbb{R}$, which is further identified with $\mathbb{R}^n$. Observe that $\mathcal{F}_A$ can be also regarded as a fan in $\mathbb{Q}^n$. However, in the latter case, we explicitly specify the ambient space.
Before we state an important connection between the behavior of the $\tau$-tilting fan and our study of bricks, let us recall that a fan $\mathcal{G}$ in $\mathbb{R}^n$ is said to be \emph{complete} if each element of $\mathbb{R}^n$ lies in some cone of $\mathcal{G}$. This motivates the following definition.

\begin{definition}
An algebra $A$ is called \emph{$g$-finite} if the $\tau$-tilting fan $\mathcal{F}_A$ is complete in $\mathbb{R}^n$.
\end{definition}

The next proposition shows the connection between the more geometric notion of $g$-finiteness and the algebraic notion of brick-finiteness.

\begin{prop}[\cite{DIJ}, \cite{As2}]\label{Prop: g-finite equ. brick-finite}
An algebra $A$ is $g$-finite if and only if it is brick-finite.
\end{prop}

From the above proposition it is immediate that $A$ is a brick-infinite algebra if and only if there are vectors in $\mathbb{R}^n$ that do not belong to the $\tau$-tilting fan $\mathcal{F}$. However, we note that such a vector, \emph{a priori}, does not need to be rational. In fact, this motivates one of the open problems that have been posed in recent years. However, to put this in context and explain the connections to our study of bricks, we first recall another notion of finiteness that has been considered more recently.

\begin{definition}\label{Def: E-finite}
An algebra $A$ is called \emph{$E$-finite} if for every $\tau$-regular component $\Z \in \Irr(A)$ we have $c(\Z)=0$. 
\end{definition}

Observe that $E$-finite algebras were introduced in \cite[Definition 6.3]{AI}, where this notion was defined in a slightly different language. However, by Theorem \ref{Plamondon's result} and some further standard arguments, it turns out that the definition given above is equivalent to the original definition. We particularly use this (equivalent) definition of $E$-finiteness because it allows us to better highlight the connections to the study of irreducible components of the algebra under consideration, as discussed in the preceding sections. Moreover, we use this notion to state (a reformulation of) a question that has been posed on the behavior of the $\tau$-tilting fan of algebras. More specifically, in \cite{De} Demonet originally asked if for every brick-infinite algebra $A$, there always exists a rational ray in the ambient space $\mathbb{R}^n$ which falls outside the $\tau$-tilting fan $\mathcal{F}_A$. 
Since this question is still open and is widely believed to hold, we state it as a conjecture.

\begin{conjecture}\label{Demonet's Conj}
If $A$ is $E$-finite, then $A$ is brick-finite.
\end{conjecture}
  
\medskip

In Figure \ref{Fig: Known and unknown implications}, we have summarized a sequence of known implications and open conjectures. For more details on the new results on these implications, see \cite{MP3} and references therein.

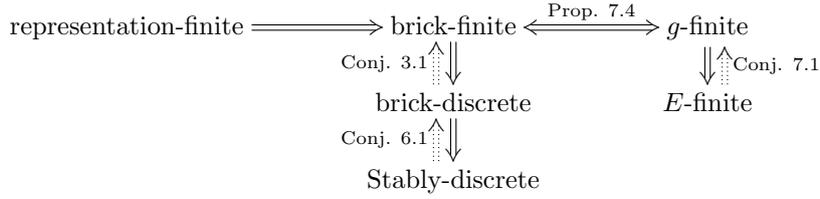
\begin{figure}
\begin{align*}
\xymatrix@R=1.5em@C=4em{
\text{representation-finite}\ar@{=>}[r]&\text{brick-finite}\ar@{<=>}[r]^-{\text{Prop. \ref{Prop: g-finite equ. brick-finite}}}\ar@{=>}[d]^-{\text{}}&\text{$g$-finite}\ar@{=>}[d] \\
 &\text{brick-discrete}\ar@{=>}[d]^-{\text{}} \ar@{=>}@<1.5ex>@{.>}[u]^-{\text{Conj. \ref{2nd bBT Conj.}}}& \text{$E$-finite}\ar@{=>}@<-1.5ex>@{.>}[u]_-{\text{Conj. \ref{Demonet's Conj}}} \\
 &\text{Stably-discrete} \ar@{=>}@<1.5ex>@{.>}[u]^-{\text{Conj. \ref{CKW Conj.}}} & } 
\end{align*}
\caption{Solid arrows indicate that the implication is known in full generality, whereas the unknown (conjectural) implications are specified by dotted arrows.}\label{Fig: Known and unknown implications}
\end{figure}

\subsection{Tameness}
For almost every new notion of finiteness that has appeared in representation theory, there is often a corresponding notion of tameness. In the following we only briefly recall those which relate to the scope of these notes.

\begin{definition}
An algebra $A$ is called \emph{brick-tame} if, for each dimension vector $\mathbf{d}\in \mathbb{Z}^n_{\geq 0}$,
there exist $A-k[x]$-bimodules $M_1, M_2, \cdots M_r$ which are free and of finite rank as right $k[x]$-modules and almost all bricks of dimension vector $\mathbf{d}$ are of the form $M_i\otimes_{k[x]} k[x]/\langle x-\lambda\rangle$, for some $1\leq i \leq r$.
\end{definition}

We remark that the notion of brick-tameness was first introduced in \cite{BD} and ultimately in a more general setting. A reformulation of this notion which is closer to the one given above appeared in \cite{CC}, where the authors showed some interesting results for some algebras. 
One immediately observes that in the more modern notion of brick-tameness, bricks replace the role of indecomposable modules in the classical notion of tameness. In fact, we remark that every tame algebra is brick-tame, but the converse is not necessarily true. For example, every wild local algebra is brick-finite, and therefore brick-tame. As a less immediate example, one can show that the following algebra $A=kQ/I$ is wild and brick-infinite, but it is brick-tame (for more details, see \cite{BD}, \cite{CC} and references therein). Here, the quiver $Q$ is given by
$$
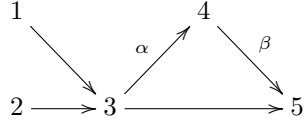
\xymatrix{1\ar[dr] & & 4 \ar[dr]^\beta & \\ 2 \ar[r] & 3 \ar[rr] \ar[ur]^\alpha &&5}$$
and the admissible ideal is $I = \langle \beta\alpha \rangle$.

\medskip

We now recall another modern notion of tameness, which is defined with respect to the behavior of the $\tau$-tilting fan of the algebra under consideration.

\begin{definition}
An algebra $A$ is called \emph{$g$-tame} if the $\tau$-tilting fan $\mathcal{F}_A$ is dense in $\mathbb{R}^n$, that is, the closure of the union of all cones of the fan equals $\mathbb{R}^n$.
\end{definition}

As shown in \cite{PY}, every tame algebra is $g$-tame, however the converse is not necessarily true. One can again consider wild local algebras as an immediate non-example. For less immediate non-examples and further details on $g$-tame algebras, we refer to \cite{PY}.

\medskip 

To recall yet another related notion of tameness recently introduced in the literature, we first need to introduce and recall some results to motivate the notion. In that regard, we let  $K^{[-1,0]}(\rm projA)$ denote the full subcategory of $K^b(\rm projA)$ which consists of complexes in degree $-1$ and $0$. Moreover, for each $g$-vector $\theta \in K_0(\rm projA)$, recall that we let $\theta$ = $\theta^+ - \theta^-$, where $\theta^+ = (\theta'_i)$ such that $\theta_i' := {\rm max}\{0, \theta_i\}$, and $\theta^-=(\theta''_i)$ given by $\theta''_i := -{\rm min}\{0, \theta_i\}$. 
In particular, $\theta^+, \theta^- \in \mathbb{Z}^n_{\geq 0}$ and they do not share nonzero coefficients.
As before, for $\theta \in K_0(\rm projA)$, we consider $P(\theta^+) := \oplus_{i=1}^n P_i^{(\theta^+)_i}$ and $P(\theta^-) := \oplus_{i=1}^n P_i^{(\theta^-)_i}$.

\medskip

Now, we are ready to recall the following notion of indecomposability.

\begin{definition}
A $g$-vector $\theta \in K_0(\rm projA)$ is called \emph{indecomposable} if $\Hom(P(\theta^-), P(\theta^+))$ has a dense set of indecomposables in $K^b(\rm projA)$.
\end{definition}

We remark that, for an algebra $A$ of rank $n$, each indecomposable $g$-vector $\theta \in K_0(\rm projA)$ belongs to the set
$\{g_{\Z} \mid \Z \text{ is an indecomposable $\tau$-regular component} \} \cup \{-[P_i]\mid i=1,2,\ldots, n\}$.
As we will observe below, indecomposable $g$-vectors play a crucial role in some new notions of tameness. However, we first need to recall some more notations and terminology. 
In particular, using the same notations as above, for $\theta_1$ and $\theta_2$ in $K_0(\rm projA)$, we let
$$E(\theta_1,\theta_2):= \min \Big\{ \dim_k \Hom(X,Y[1]) \mid X: P(\theta_1^-) \to P(\theta_1^+) \text{ and } Y: P(\theta_2^-) \to P(\theta_2^+) \Big\}.$$
Then, two $g$-vectors $\theta_1$ and $\theta_2$ in $K_0(\rm projA)$ are said to be \emph{compatible} if $E(\theta_1,\theta_2)=0=E(\theta_2,\theta_1)$. Now, we can recall the following important notion of decomposition, as studied by Derksen and Fei in \cite{DF}.

\begin{definition}
For a $g$-vector $\theta \in K_0(\rm projA)$, the \emph{canonical decomposition} of $\theta$ is given by the expression $\theta=\theta_1+\theta_2+\cdots+\theta_r$ such that there exists a nonempty dense set $\mathcal{O} \subseteq \Hom(P(\theta^-),P(\theta^+))$ where each $X \in \mathcal{O}$ decomposes into indecomposables in $K^b(\rm projA)$ whose $g$-vectors are $\theta_1,\theta_2,\cdots,\theta_r$.
\end{definition}

The following theorem gives a complete description of the canonical decomposition of a given $g$-vector.

\begin{theorem}\cite{DF} 
For a $g$-vector $\theta \in K_0(\rm projA)$, the decomposition $\theta=\theta_1+\theta_2+\cdots+\theta_r$ is the canonical decomposition of $\theta$ if and only if the following hold:
\begin{itemize}
    \item Each $\theta_i$ is indecompoable;
    \item For each $1\leq i<j \leq r$, the $g$-vectors $\theta_i$ and $\theta_j$ are compatible.
\end{itemize}
\end{theorem}

Now we recall two new notions of tameness that have been considered in recent years.

\begin{definition}
Let $A$ be an algebra. Then,
\begin{enumerate}
    \item A is called \emph{$E$-tame} if, for each $g$-vectors $\theta \in K_0(\rm projA)$, we have $E(\theta,\theta)=0$.
    \item $A$ is said to be \emph{stably-tame} if,  for each $g$-vectors $\theta \in K_0(\rm projA)$ and every $\Z \in \Irr(A)$, the dimension of $\mathcal{M}^{\theta-s}(\mathcal{Z})$ is either $0$ or $1$.
\end{enumerate}
\end{definition}

In Figure \ref{Fig: Known and unknown implications of tameness}, we have summarized a set of implications and open questions on various notions of tameness discussed above.

\medskip

Let us mention the following result, which was first proved by Pfeifer in \cite{Pf}, and was also proved later in \cite{MP2}.

\begin{prop}\label{Prop: E-tame but not E-finite}
    Let $A$ be $E$-tame but not $E$-finite. Then $A$ admits an infinite family of bricks, which are all $\theta$-stable for some $\theta \in K_0(\proj A)$.
\end{prop}

\begin{figure}
\begin{align*}
\xymatrix@R=1.5em@C=4em{
\text{tame}\ar@{=>}[r]^-{\text{\cite{CC}}}\ar@{=>}[dd]_-{\text{\cite{PY}}}\ar@{=>}[rdd]^-{\text{\cite{AI}}}&\text{brick-tame}\ar@{=>}[r]^-{\text{Def}}\ar@{=>}@<1.5ex>@{.>}[dd]&\text{stably-tame} \ar@{=>}@<1.5ex>@{.>}[l]\\ 
& & \\
\text{$g$-tame}\ar@{=>}@<1.5ex>@{.>}[r] &\text{$E$-tame}\ar@{=>}@<1.5ex>@{.>}[l]\ar@{=>}@<1.5ex>@{.>}[uu] & } 
\end{align*}
\caption{Solid arrows indicate that the implication is known. The unknown implications are specified by dotted arrows.}\label{Fig: Known and unknown implications of tameness}
\end{figure}
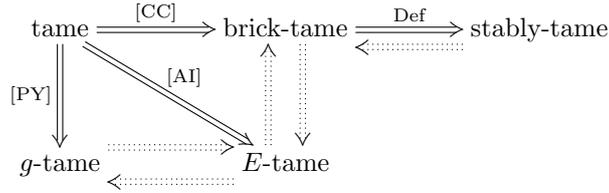

\medskip

Now, we are ready to state the following theorem on tame algebras, which strengthens the connections between various classical and modern notions of finiteness. Furthermore, it highlights the significance of components of the Auslander-Reiten quiver of algebras in the study of such properties and the open conjectures listed above. 

\medskip

In the following, by $\Gamma_A$ we denote the Auslander-Reiten quiver of $A$. 
We particularly recall that a component $\mathcal{C}$ of $\Gamma_A$ is said to be \emph{generalized standard} if $\rad^{\infty}(X,Y)=0$, for all $X$ and $Y$ in $\mathcal{C}$. In particular, every preprojective (similarly, preinjective) component is known to be generalized standard, but there are many generalized standard components which are neither preprojective nor preinjective. 
It is well known that the components of $\Gamma_A$ play a decisive role in the systematic study of representation theory of algebras, and various properties of them have been studied for several decades. More specifically, for more details on the generalized standard components and their properties, we refer to \cite{SS} and references therein.

\begin{theorem}\label{Thm: equivalences for tame algebras}
Let $A$ be a tame algebra. The following are equivalent:
\begin{enumerate}
    \item $A$ is brick-discrete;
    \item $A$ is $E$-finite;
    \item $A$ is stably-discrete;
    \item $\Gamma_A$ has no generalized standard components;
    \item $\Gamma_A$ has no infinite family of Hom-orthogonal tubes with quasi-simples of the same dimension.
\end{enumerate}
\end{theorem}
\begin{proof}[Sketch of proof]
We observe that the equivalence of $(1)\leftrightarrow (3)$ is shown in \cite{CKW}. 
Moreover, if $A$ is tame, there is a correspondence between $\Big\{\text{indecomposable $\tau$-regular components $\Z$ with $c(\Z)=1$}\Big\}$ and $\Big\{\text{$1$-parameter families of bricks}\Big\}$, for instance, see \cite{PY}. This implies the equivalence of $(1), (2)$ and $(3)$. 
Finally, the equivalence of $(4)$ and $(5)$ to the other parts is shown in \cite[Corollary D]{MP3}.     
\end{proof}

We finish this section with some remarks on the above theorem and how it fits into the larger framework of the problems treated in our work and some other researchers. 

\begin{remark}\label{Rem: 2nd bBT open for tame}
With regard to Theorem \ref{Thm: equivalences for tame algebras}, we remark that although the 2nd bBT (see Conjecture \ref{2nd bBT Conj.}) is verified for some important families of tame algebras, the conjecture is still open for arbitrary tame algebras. More specifically, it is still not known whether the above $5$ conditions are also equivalent to $A$ being brick-finite. Some new results on this problem are discussed in the following section. 

\medskip

Meanwhile, we note that the above problems on the behavior of bricks are also closely related to the study of infinite dimensional brick, as shown in the important work of Sentieri \cite{Se}, as well as to the generic bricks, formally introduced by the authors in \cite{MP1}. We do not discuss these connections in the current survey. However, as noted earlier, the study of generic bricks over tame algebras is more tractable. In fact, it provides powerful tools in the treatment of several open brick-Brauer-Thrall Conjectures, which further leads to conceptual linkages between the scope of this manuscript and study of spectrum of algebras. For more details and some new results on generic bricks, see \cite{BPS}, \cite{MP1}, \cite{MP3}, and references therein. 

\medskip

Finally, we note that, in addition to some earlier work on $E$-finite and $E$-tame algebras in \cite{MP2} and \cite{Pf} primarily motivated by the study of bricks and some related studies in \cite{AI} and \cite{DF}, recently there have been some further results of the equivalence of $g$-tameness and $E$-tameness (and, respectively, on $g$-finiteness and $E$-finiteness) for some special families of algebras. In particular, in \cite{HY} some new studies on the equivalence of these notions are conducted for tame Jacobian algebras. Such results directly relate to Theorem \ref{Thm: equivalences for tame algebras} (and, therefore, Conjectures \ref{2nd bBT Conj.} and \ref{Demonet's Conj}), as well as our new results on tame algebras discussed in the following section. Since these connections are quite straightforward, we do not give an elaborate description of implications and leave the details to the reader.
\end{remark}

\vskip 0.5cm

\section{Projective dimension of bricks and bBT conjectures}\label{Sec: Pojective dimension of bricks and open conjectures}
It is well known that for any classification result on the behavior of indecomposable modules, tame algebras are generally better understood. However, despite some important results on their bricks that have already appeared in the literature, many modern conjectures still remain open for arbitrary tame algebras (see Remark \ref{Rem: 2nd bBT open for tame}). Meanwhile, it is known that many of the recent conjectures on bricks are in fact equivalent over tame algebras (see \cite[Section 2.2 and Corollary D]{MP3}). Thus, new treatments of brick-infinite tame algebras can provide a fresh impetus to the study of a range of closely related problems (for details, see \cite{MP3} and references therein). In this section, we use some of our earlier results and recent observations to propose a new and equivalent simultaneous treatment of several open conjectures on tame algebras. More specifically, we show that the study of some important problems, the so-called \emph{brick-Brauer-Thrall conjectures} (\textbf{bBT Conjectures}) listed in \cite{MP3}, closely relates to the study of the projective dimension of bricks. 

\medskip

To outline the content of this section and put our main results into perspective, we first recall a notion of minimality that allows us to naturally reduce the study of the bBT conjectures to some particular algebras of interest. Below, we work in the same setting and under the same assumptions listed in the beginning of the article. Recall that $\brick(A)$ denotes the set of (isomorphism classes of) bricks in $\modu A$.

\begin{definition}
An algebra $A$ is said to be \emph{minimal brick-infinite} if $\brick(A)$ is an infinite set, but $\brick(A/J)$ is finite, for every nonzero (two-sided) ideal $J$ in $A$. 
\end{definition}

We remark that minimal brick-infinite algebras were originally introduced through the lens of $\tau$-tilting theory, and they can be viewed as the modern counterparts of the classical minimal representation-infinite algebras (see \cite[Chapter V.I]{Mo1} and \cite[Section 3]{Wa}). These algebras are further studied in \cite{MP4}, where they are shown to manifest many interesting properties, including in the classical tilting theory. We particularly note that the study of the 2nd bBT conjecture can be reduced to minimal brick-infinite algebras (for details and nontrivial reductions, see \cite{MP4}). 
We also note that, thanks to the result of \cite{Se} on infinite-dimensional bricks, an alternative characterization of minimal brick-infinite algebras can be given in terms of such bricks: \emph{An algebra $A$ is minimal brick-infinite if and only if it admits an infinite-dimensional brick, and every such brick is faithful}. 
As remarked earlier, in the current work, our focus remains on the study of finite dimensional bricks.

\medskip

In this section, we mainly focus on those minimal brick-infinite algebras which are tame, and we derive some new results on this important family. More specifically, for each minimal brick-infinite tame algebra $A$, we reduce the 2nd bBT conjecture to the study of a particular torsion pair in $\modu A$. This ultimately implies an interesting condition on the projective dimension of bricks in $\modu A$.

\medskip

Before we present some useful facts, let us fix some notation. In particular, for an algebra $A$, we set $\mathcal{T}_I:=\Filt(\Fac(DA))$, where $D(-)=\Hom_k(-,k)$ is the standard dual. Namely, $\mathcal{T}_I$ denotes the smallest torsion class containing $DA$; this torsion class is equivalently determined as the smallest torsion class in $\tors(A)$ containing all injectives in $\modu A$. The corresponding torsion pair is denoted by $(\mathcal{T}_I, \mathcal{F}_I)$. 
Moreover, for a pair of torsion classes $\mathcal{T}_1$ and $\mathcal{T}_2$ in $\tors(A)$, define 
$$[\mathcal{T}_1, \mathcal{T}_2]:= \{\mathcal{T}\in \tors(A) \,|\, \mathcal{T}_1 \subseteq \mathcal{T} \subseteq \mathcal{T}_2 \},$$
which is the interval of torsion classes in $\tors(A)$ specified by $\mathcal{T}_1$ and $\mathcal{T}_2$. Finally, recall that $\mathcal{T} \in \tors(A)$ is called \emph{faithful} if for each nonzero ideal $J$ is $A$, there exists $M \in \mathcal{T}$ such that $JM\neq 0$.

\begin{lemma}\label{Lem: faithful torsion class}
For an algebra $A$, with the same notation as above, the following are equivalent:
\begin{enumerate}
    \item $\mathcal{T}$ in $\tors(A)$ is faithful.
    \item $\mathcal{T}_I \subseteq \mathcal{T}$.
\end{enumerate}
\end{lemma}
\begin{proof}
    We first note that a torsion class $\mathcal{T}$ is faithful if and only if it contains a faithful module $M$. Second, a module $M$ is faithful if and only if there exists a positive integer $r$ and an epimorphism $M^r \to DA$. This shows that if $\mathcal{T}$ is faithful, then it has to contain $DA$. Conversely, if $\mathcal{T}$ contains $DA$, then it is obviously faithful.
\end{proof}

The following result gives a better understanding of the lattice structure of torsion classes of minimal brick-infinite algebras. To put the result in perspective, note that for any brick-infinite algebra $A$, from Proposition \ref{Prop: tau-tilting-finiteness} it follows that $\tors(A)$ is an infinite lattice, and, moreover, infinitely many $\mathcal{T}\in \tors(A)$ are not functorially finite.

\begin{proposition}\label{Prop: lattice of torsion classes of min-brick-inf algs}
Let $A$ be a minimal brick-infinite algebra. With the same notation as above, we have
\begin{enumerate}
    \item $[\mathcal{T}_I, \modu A]$ consists of all faithful torsion classes of $\modu A$. It is a complete infinite sublattice of $\tors(A)$ which contains every non-functorially finite torsion class $\mathcal{T} \in \tors(A)$.
    \item $\tors(A) \setminus [\mathcal{T}_I, \modu A]$ is a finite set and each torsion class in it can be obtained from $0$ by applying finitely many right mutations. In particular, the Hasse diagram of $\tors(A) \setminus [\mathcal{T}_I, \modu A]$ is connected and consists only of unfaithful functorially finite torsion classes.
\end{enumerate}
\end{proposition}
\begin{proof}
The first assertion of part (1) follows from Lemma \ref{Lem: faithful torsion class}. Moreover, being an interval in the complete lattice $\tors A$ makes $[\mathcal{T}_I, \modu A]$ a complete sublattice of $\tors A$. Now, if $\mathcal{T}$ is not functorially finite, then by \cite[Theorem 1.3]{DIJ}, there exists an infinite chain of torsion classes $0 \subsetneq  \mathcal{T}_1 \subsetneq \mathcal{T}_2 \subsetneq \cdots $ in $[0, \mathcal{T}]$, where each $\mathcal{T}_i$ is functorially finite. In particular, $\mathcal{T}_i=\Fac(M_i)$, for a support $\tau$-tilting module $M_i$ in $\modu A$. Meanwhile, since $A$ is minimal brick-infinite, almost all support $\tau$-tilting modules are tilting, hence faithful (see \cite[Theorem 4.2 and Corollary 4.4]{MP4}). Again, by Lemma \ref{Lem: faithful torsion class}, we have $\mathcal{T}_I \subseteq \mathcal{T}_i$, for almost all $i \in \mathbb{Z}_{>0}$, which shows that $\mathcal{T}_I \subseteq \mathcal{T}$.

To prove $(2)$, first observe that an argument similar to the proof of $(1)$ shows that any torsion class in $\tors(A) \setminus [\mathcal{T}_I, \modu A]$ is obtained from $0$ by applying finitely many right mutations, each of which being unfaithful. It follows that each such torsion class is generated by an unfaithful support $\tau$-tilting module. Since $A$ is minimal brick-infinite, we only have finitely many such modules (see \cite[Theorem 4.2]{MP4}).
\end{proof}

In classical representation theory, where the focus is on the study of indecomposable modules, an algebra $A$ is sometimes said to be \emph{mild} if every proper quotient of $A$ is representation-finite; that is, $A$ is either representation-finite or else $A$ is minimal representation-infinite. One can analogously define the notion of \emph{brick-mild} algebras as those for which every proper quotient is brick-finite. In light of the preceding proposition, we can give a characterization of the brick-mild algebras in terms of the lattice of torsion classes. 
More precisely, the following corollary gives a conceptual characterization of brick-mild algebras.

\begin{corollary}
An algebra $A$ is brick-mild if and only if $\tors(A) \setminus [\mathcal{T}_I, \modu A]$ is finite.
\end{corollary}
\begin{proof}
Let $A$ be brick-mild. If $A$ is brick-finite, the assertion is trivial, and if $A$ is minimal brick-infinite, the implication follows from  part $(2)$ of Proposition \ref{Prop: lattice of torsion classes of min-brick-inf algs}. 
Conversely, assume that $\tors(A) \setminus [\mathcal{T}_I, \modu A]$ is finite. If $A$ is brick-infinite and $A/J$ is brick-infinite, for a nonzero two-sided ideal $J$, then $JX=0$, for infinitely many $X\in \brick(A)$. Thus, by Lemma \ref{Lem: faithful torsion class}, for any such brick $X$, we have $T(X) \notin [\mathcal{T}_I, \modu A]$. Observe that, by Proposition \ref{Prop: MS injective}, for distinct $X, Y \in \brick(A)$, we have $T(X)\neq T(Y)$. 
This contradicts the assumption that $\tors(A) \setminus [\mathcal{T}_I, \modu A]$ is finite, and we are done.
\end{proof}

Following \cite{MP2}, an irreducible component $\mathcal{Z}$ in $\Irr(A)$ is called \emph{faithful} if there is no non-zero two sided ideal $J$ that annihilates all modules in $\mathcal{Z}$. Observe that a faithful component $\mathcal{Z}$ does not necessarily contain a faithful module. In the latter case, we say that $\Z$ is a \emph{strongly faithful} component. Evidently, a strongly faithful component is faithful. We say that a brick $X$ is \emph{weakly faithful} it it lies in an irreducible component $\mathcal{Z} \in \Irr A$ which is faithful (but not necessarily strongly faithful).

\begin{prop} \label{lem: faithful bricks}
    Let $A$ be any algebra. Then for all but possibly one weakly faithful brick $X$, we have $\Hom_A(DA,X)=0$ and consequently, ${\rm pd}_A (\tau^{-1}X) \le 1$.
\end{prop}

\begin{proof}
    Assume that $X$ is a faithful brick and consider the torsion class $T(X)$. Since $X$ is faithful, by Lemma \ref{Lem: faithful torsion class} we have $\mathcal{T}_I \subseteq T(X)$. In particular, unless $T(X) = \mathcal{T}_I$ (which could happen for at most one brick by Proposition \ref{Prop: MS injective}), we get the inclusion $\mathcal{T}_I \subseteq T(X)_* = T(X) \cap ^\perp X$. Since $X$ is torsion-free with respect to the torsion class $T(X)_*$, we see that $X \in \mathcal{F}_I$. Therefore, $\Hom_A(DA, X)=0$.  Now, we recall that a module $Y$ has ${\rm pd}_A Y \leq 1$ if and only if $\Hom_A(DA, \tau Y)=0$; see \cite[Chap. IV, Lemma 2.7]{ASS}.

    Let now $X$ be a weakly faithful brick which is not faithful. Then there is a faithful brick component $\mathcal{Z}$ with $c(\mathcal{Z}) \ge 1$ containing $X$. It follows from \cite{As1} that $X$ can be completed to an infinite semibrick $\mathcal{X}$ with $\mathcal{X} \subseteq \mathcal{Z}$. We can further assume that $\mathcal{X} \setminus \{X\}$ is faithful.  Once again, by Proposition \ref{Prop: MS injective}, we have a strict inclusion of torsion classes
    $$T(\mathcal{X} \setminus \{X\}) \subsetneq T(\mathcal{X}),$$  where $X$ is torsion-free with respect to the smallest torsion class, which contains $\mathcal{T}_I$. This makes $X \in \mathcal{F}_I$ and the desired result follows.
\end{proof}

\begin{remark}
     When ${\rm pd}_A(\tau^{-1}X)=1$, we get that $\tau^{-1}X$ lies in some $\tau$-regular component. However, even if $X$ is in general position in an irreducible component containing it, it does not necessarily follow that $\tau^{-1}X$ is in general position. In particular, if the orbit of $X$ is dense, then $\tau^{-1}X$ need not be $\tau$-regular (i.e., $\tau^{-1}X$ is not necessarily in general position in a $\tau$-regular component).
\end{remark}

Before proving the next theorem, for simplicity we define 
$$\tau\brick(A):=\{\tau X \,|\, X \in \brick(A) \}.$$
Observe that there is a one-to-one correspondence between elements of $\tau\brick(A)$ and the non-projective bricks $X \in \brick(A)$. In particular, $A$ is brick-finite if and only if $\tau\brick(A)$ is a finite set. Moreover, from the equivalence $\tau: \underline{\modu} A \rightarrow \overline{\modu} A$ induced by the Auslander-Reiten translation, it follows that for each $X \in \tau\brick(A)$, we have $\overline{\End}_A(X)=k$.

We recall that for any algebra $A$, as before we set $\mathcal{F}_I= \mathcal{T}_I^{\perp} = (DA)^\perp$. 

\begin{theorem} \label{Thm: 2nbBT equivalence for tame}
Let $A$ be a minimal brick-infinite tame algebra. Then, the following are equivalent:
\begin{enumerate}
    \item $\tau\brick(A)\cap \mathcal{F}_I$ is an infinite set.
    \item For some $d\in \mathbb{Z}_{>0}$, there is an infinite family of bricks of dimension $d$ in $\brick(A)$.
\end{enumerate}
\end{theorem}
\begin{proof}
To prove $(1)\rightarrow (2)$, let $\tau\brick(A)\cap \mathcal{F}_I$ be an infinite set. Then, by \cite[IV. Lemma 2.7]{ASS}, for each $\tau B$ in $\tau\brick(A)\cap \mathcal{F}_I$, we have $\pd B\leq 1$. Consequently, by Proposition \ref{Prop: vanishing Ext and orbit openennes}, we have that $\Ext^1_A(B,B)=0$ if and only if $\mathcal{O}_B$ is open in any irreducible component of $\Irr(A)$ that contains $B$. 
Therefore, we may assume that all such bricks $B$ are rigid (as otherwise, for one of them $\mathcal{O}_B$ will not be open, and hence we immediately get the desired result). Using the Auslander-Reiten formula, this implies that all such bricks $B$ are in fact $\tau$-rigid. Hence, we have an infinite family of $\tau$-rigid bricks over a minimal brick-infinite algebra. By \cite[Proposition 1.5]{MP2}, we get an (integral) $g$-vector $\theta \in K_0(\proj A)$ that lies outside of the $\tau$-tilting fan $\mathcal{F}$. This implies that $A$ satisfies the assumptions of \cite[Proposition 8.1]{MP2}, and hence it admits an infinite family of bricks of the same dimension.

To show $(2)\rightarrow (1)$, let $\{X_i\}_{i\in I}$ be an infinite subset of $\brick(A)$ with $\dim_k(X_i)=d$.  This yields a brick component $\mathcal{Z}$ with $c(\mathcal{Z})=1$ in $\Irr (A)$ which has to be faithful since $A$ is minimal brick-infinite. It follows from \cite[Corollary 5.6]{MP2} that almost all $X_i$ have projective dimension one. By Crawley-Boevey's result on tame algebras \cite{C-B}, we also get that almost all $X_i$ are $\tau$-homogeneous, that is $\tau X_i=X_i$. Consequently, because $\pd X_i=1$ for infinitely many $X_i$, we have $\Hom(DA,\tau X_i)=0$. This implies $X_i=\tau X_i \in \tau \brick(A) \cap \mathcal{F}_I$, for infinitely many $X_i$.
\end{proof}

\begin{remark}
From \cite[Corollary D]{MP3} and \cite[Theorem 5.4]{MP4} it follows that, to settle Conjectures \ref{2nd bBT Conj.} and \ref{CKW Conj.} for all tame algebras, we can additionally restrict ourselves to those minimal brick-infinite algebras over which almost all bricks are faithful. Equivalently, in light of the above theorem, we can assume $A$ is a minimal brick-infinite tame algebra such that all but finitely many modules in $\brick(A)$ are faithful. And to settle the aforementioned conjectures for all tame algebras, it suffices to show that for every such algebra $A$, the set $\tau\brick(A)\cap \mathcal{F}_I$ is infinite.
\end{remark}

In the following corollary, we observe that for tame algebras, bricks of projective dimension $1$ play a decisive role in the study of the 2nd bBT conjectures (see Conjecture \ref{2nd bBT Conj.}).
For simplicitly, we define
$$\text{brick}^{\leq 1}(A):=\{X\in \brick(A) \, | \, \pd X\leq 1 \}.$$

\begin{corollary}\label{Cor: proj. dim of tame min-brick-inf algs}
For any minimal brick-infinite tame algebra $A$, the following are equivalent:
\begin{enumerate}
    \item $A$ satisfies the $2nd$ bBT conjecture.
    \item $A$ admits a $1$-parameter family of bricks of projective dimension $1$.
    \item $\text{brick}^{\leq 1}(A)$ is an infinite set.
 \end{enumerate}
 Therefore, a brick-infinite tame algebra satisfies the $2nd$ bBT conjecture if and only if a quotient of it admits a 1-parameter family of bricks of projective dimension one.
\end{corollary}
\begin{proof}
Let $A$ be minimal brick-infinite tame. Recall that $\tau X \in \tau\brick(A)\cap \mathcal{F}_I$ if and only if $\pd_A X \le 1$. The proof of Theorem \ref{Thm: 2nbBT equivalence for tame} shows that when $\tau\brick(A)\cap \mathcal{F}_I$ is infinite, there is a 1-parameter family $\{X_i\}_{i \in I}$ of $\tau$-homogeneous bricks. By minimality, these bricks are weakly faithful and it follows from Proposition \ref{lem: faithful bricks} and $\tau$-homogeneity that these bricks all lie in $\tau\brick(A)\cap \mathcal{F}_I$.  This shows that $(1), (2)$ and $(3)$ are all equivalent.

For the last assertion, observe that a brick-infinite tame algebra, say $\Lambda$, satisfies the $2nd$ bBT conjecture if and only if $\Lambda$ has a brick component $\mathcal{Z}$ with $c(\mathcal{Z})=1$. Let $J = {\rm Ann}(\mathcal{Z})$ and note that the algebra $\Lambda/J$ still has $\mathcal{Z}$ as a brick component. However, the 1-parameter family of bricks in $\mathcal{Z}$ are now weakly faithful $\tau_{\Lambda/J}$-homogeneous, hence of projective dimension 1 over the quotient algebra $\Lambda/J$. 
\end{proof}

To put the above corollary into a larger perspective, we finish this section with some remarks on these new results and some important links to the earlier parts of this article and older work.

\begin{remark}
Let us highlight the connections between Corollary \ref{Cor: proj. dim of tame min-brick-inf algs} and the main results of the preceding section, particularly with Proposition \ref{Prop: E-tame but not E-finite} and Theorem \ref{Thm: equivalences for tame algebras}. These results, together, lead to some further interesting equivalences on (minimal brick-infinite) tame algebras, which we leave to the reader. 
Meanwhile, we emphasize that, from \cite[Theorem 5.4]{MP4} it follows that every minimal brick-infinite tame algebra with infinitely many unfaithful bricks admits a $1$-parameter family of bricks of projective dimension $1$, hence it satisfies all the above-mentioned equivalent conditions.

\medskip

On the other hand, it is known that tame concealed algebras are minimal brick-infinite, and they admit a $1$-parameter family of bricks of projective dimension $1$. By Corollary \ref{Cor: proj. dim of tame min-brick-inf algs}, and in light of the $2nd$ bBT conjecture, it is expected that all minimal brick-infinite tame algebras admit a similar property, that is, to have a $1$-parameter family of bricks of projective dimension $1$. Meanwhile, observe that, unlike concealed algebras, a minimal brick-infinite tame algebra can have infinite global dimension (for explicit examples, see the generalized barbell algebras studied in \cite{Mo1} and \cite{MP1}).
\end{remark}

\medskip

\newpage

\textbf{Acknowledgments.} The authors would like to thank the organizers of the Autumn School-Conference on New Developments in Representation Theory of Algebras, held between November 18-30 of 2024, at Okinawa Institute of Science and Technology (OIST), Japan. We also thank the participants for many interesting discussions and feedback during and after the event. 
KM was supported by Early-Career Scientist JSPS Kakenhi grant number 24K16908. CP was supported by the Natural Sciences and Engineering Research Council of Canada (RGPIN-2018-04513) and by the Canadian Defence Academy Research Programme. 

\medskip

On a more personal note, the first-named author (KM) wishes to express his deep gratitude to Prof. O. A. S. Karamzadeh, for some insightful guidance on mathematical thinking which has profoundly influenced his approach towards mathematics and beyond. In fact, analyzing his initial interest in the main conjecture treated in this manuscript (see Conjecture \ref{2nd bBT Conj.}), the genesis of KM’s work on this problem can be traced back to a formative lesson during his Master's studies under Prof. Karamzadeh: ``wonder if an interesting result can be formulated as an ‘if, and only if’ statement, and in a well-motivated way.". 
As a token of appreciation for this valuable insight -- along with many other lasting contributions -- and in recognition of the pedagogical and aesthetic value of such formulations in mathematics, as recently communicated to a wider audience in [Ka] with several illustrative examples, KM would also like to present possibly the simplest (yet equivalent) articulation of his original conjecture, as follows:

\textbf{2nd bBT Conjecture} (\cite[Conjecture 6.0.1]{Mo1}):
\emph{An algebra $A$ admits infinitely many (non-isomorphic) bricks if, and only if, $\modu A$ contains infinitely many (non-isomorphic) bricks of the same dimension.} 

Let us observe that understanding the above formulation, expressed in this elementary language, requires only basic knowledge of module theory and linear algebra -- just enough to understand the definitions of a \emph{brick} and the dimension of a vector space. This choice also reflects our strong preference for the name ``Second brick-Brauer-Thrall Conjecture", as opposed to alternative names used in the literature. This hopefully aligns with the spirit of Prof. Karamzadeh’s perspective, who once remarked: ``Mathematics is well-defined doodling, and this doodling is never complete unless it can be shared by laymen.”.

\end{document}